\documentclass[12pt]{amsart}
\usepackage[utf8]{inputenc}
\usepackage{amsmath}
\usepackage{amssymb}
\usepackage{mathrsfs}

\usepackage[colorlinks]{hyperref}
\PassOptionsToPackage{colorlinks}{hyperref}
\usepackage[svgnames, dvipsnames, usenames]{xcolor}
\hypersetup{urlcolor = RedViolet, linkcolor = RoyalBlue, citecolor = ForestGreen}
\usepackage[capitalize]{cleveref}

\def\ppen{\penalty300 }
\let\col=\colon
\def\colon{\col\ppen}

\usepackage{tikz}
\usepackage{algorithm}
\usepackage{algpseudocode}

\theoremstyle{plain} 
\newtheorem{thm}{Theorem}[section]
\newtheorem{prop}[thm]{Proposition}
\newtheorem{lem}[thm]{Lemma}
\newtheorem{cor}[thm]{Corollary}
\newtheorem{cand}[thm]{Candidates}
\theoremstyle{definition}
\newtheorem{defn}[thm]{Definition}
\newtheorem{rem}[thm]{Remark}
\newtheorem{ex}[thm]{Example}
\newtheorem{prob}[thm]{Problem}

\renewcommand{\theta}{\vartheta}
\renewcommand{\phi}{\varphi}
\renewcommand{\epsilon}{\varepsilon}
\renewcommand{\subset}{\subseteq}
\renewcommand{\supset}{\supseteq}

\newcommand{\N}{\mathbb N}
\newcommand{\Z}{\mathbb Z}

\newcommand{\C}{\mathbb C}
\newcommand{\T}{\mathbb T}
\DeclareMathOperator{\Mor}{Mor}

\DeclareMathOperator{\spanlin}{span}
\DeclareMathOperator{\Lrot}{Lrot}
\DeclareMathOperator{\Rrot}{Rrot}

\newcommand{\Cat}{\mathscr{C}}
\newcommand{\Kat}{\mathscr{K}}
\newcommand{\id}{\mathrm{id}}

\newcommand{\GL}{\mathrm{GL}}
\newcommand{\Part}{\mathscr{P}}
\newcommand{\Partlin}{\mathbf{Part}}
\newcommand{\PartRed}{\mathbf{PartRed}}
\newcommand{\Pair}{\mathbf{Pair}}
\newcommand{\NCPair}{\mathbf{NCPair}}
\newcommand{\NCPart}{\mathbf{NCPart}}
\newcommand{\EvenPart}{\mathbf{EvenPart}}

\newcommand{\Mat}{\mathbf{Mat}}
\newcommand{\dred}{_{\delta\mathchar `\-{\rm red}}}

\newcommand{\Tc}{\mathcal{T}}
\newcommand{\Dc}{\mathcal{D}}
\newcommand{\Jc}{\mathcal{J}}
\newcommand{\Bc}{\mathcal{B}}
\newcommand{\Pc}{\mathcal{P}}
\newcommand{\Vc}{\mathcal{V}}

\newcommand{\la}{\mathsf{a}}
\newcommand{\lb}{\mathsf{b}}
\newcommand{\lc}{\mathsf{c}}

\DeclareMathOperator{\tens}{tens}
\DeclareMathOperator{\contr}{contr}
\DeclareMathOperator{\rot}{rot}

\DeclareMathOperator{\Pol}{Pol}

\newcommand{\ethine}{
\Partition{
\Pblock 1to0.6:4,5
\Pblock 1to0.2:3,6
\Pblock 1to-0.2:2,7
\Pline (1,1) (1,-0.2)
\Pline (8,1) (8,-0.2)
}}

\newcommand{\cyclopropadiene}{
\Partition{
\Pblock 1to0.6:4,5
\Pblock 1to0.2:3,6
\Pblock 1to0.6:8,9
\Pblock 1to0.2:7,10
\Pblock 1to-0.2:2,11
\Pline (1,1) (1,-0.2)
\Pline (12,1) (12,-0.2)
}}

\newcommand{\cyclobutatriene}{
\Partition{
\Pblock 1to0.6:4,5
\Pblock 1to0.2:3,6
\Pblock 1to0.6:8,9
\Pblock 1to0.2:7,10
\Pblock 1to0.6:12,13
\Pblock 1to0.2:11,14
\Pblock 1to-0.2:2,15
\Pline (1,1) (1,-0.2)
\Pline (16,1) (16,-0.2)
}}

\newcommand{\cyclopentaquartene}{
\Partition{
\Pblock 1to0.6:4,5
\Pblock 1to0.2:3,6
\Pblock 1to0.6:8,9
\Pblock 1to0.2:7,10
\Pblock 1to0.6:12,13
\Pblock 1to0.2:11,14
\Pblock 1to0.6:16,17
\Pblock 1to0.2:15,18
\Pblock 1to-0.2:2,19
\Pline (1,1) (1,-0.2)
\Pline (20,1) (20,-0.2)
}}

\newcommand{\Pabaaba}{%
\Partition{
\Pblock 1to0.5:1,3
\Pblock 0to0.5:1,3
\Pline (2,1)(2,0)
}}

\newcommand{\Pabbabb}{\Partition{
\Pblock 1to0.75:1,2
\Pblock 0to0.25:2,3
\Pline (2.5,0.25)(1.5,0.75)
\Pline (1,0)(3,1)
}}

\newcommand{\Paaaaaa}{
\Partition{
\Pblock 1to0.7:1,2,3
\Pblock 0to0.3:1,2,3
\Pline (2.5,0.7)(2.5,0.3)
}}

\def\PDblock #1to#2:#3 {\PDblockA{#1}{#2}#3,,}
\def\PDblockA#1#2#3,#4,{\ifx,#4,\else
\pgfpathmoveto{\pgfpointxy{#3}{#1}}
\pgfpathlineto{\pgfpointxy{#3}{#2}}
\pgfpathlineto{\pgfpointxy{#4}{#2}}
\pgfpathlineto{\pgfpointxy{#4}{#1}}
\pgfsetdash{{\pgflinewidth}{1pt}}{0pt}
\pgfusepath{stroke}
\pgfsetdash{}{0pt}
\fihere\PblockA{#1}{#2}#4,\fi}

\newcommand{\LDabac}{\Partition{
\Psingletons 0to0.4:2,4
\PDblock 0to0.8:1,3
}}

\newcommand{\LDabcbad}{\Partition{
\Psingletons 0to0.4:3,6
\PDblock 0to1:1,5
\PDblock 0to0.7:2,4
}}

\newcommand{\Labcbad}{\LPartition{0.4:3,6}{0.7:2,4;1:1,5}}
\newcommand{\Labcbde}{\LPartition{0.4:1,3,5,6}{0.7:2,4}}
\newcommand{\Labcdae}{\LPartition{0.4:2,3,4,6}{1:1,5}}
\newcommand{\Labcdef}{\LPartition{0.4:1,2,3,4,5,6}{}}

\newcommand{\transpart}{%
\Partition{
\Pline(1,1)(1,0)
\Pline(3,1)(3,0)
\Pline(4,1)(5,0)
\Pline(5,1)(4,0)
\Pline(6,1)(6,0)
\Pline(8,1)(8,0)
\Ptext(2,0.5){$\ldots$}
\Ptext(7,0.5){$\ldots$}
}}

\newcommand{\eventranspart}{%
\Partition{
\Pline(1,1)(1,0)
\Pline(3,1)(3,0)
\Pline(4,1)(6,0)
\Pline(5,1)(5,0)
\Pline(6,1)(4,0)
\Pline(7,1)(7,0)
\Pline(9,1)(9,0)
\Ptext(2,0.5){$\ldots$}
\Ptext(8,0.5){$\ldots$}
}}

\usepackage{partmac}

\begin{document}
\title{Generating linear categories of partitions}
\author{Daniel Gromada}
\author{Moritz Weber}
\address{Saarland University, Fachbereich Mathematik, Postfach 151150,
66041 Saarbr\"ucken, Germany}
\email{gromada@math.uni-sb.de}
\email{weber@math.uni-sb.de}
\date{\today}
\subjclass[2010]{18D10 (Primary); 20G42, 68W30 (Secondary)}
\keywords{category of partitions, tensor category, non-easy quantum group}
\thanks{Both authors were supported by the collaborative research centre SFB-TRR 195 ``Symbolic Tools in Mathematics and their Application''. The second author was also supported by the ERC Advanced Grant NCDFP, held by Roland Speicher and by the DFG project ``Quantenautomorphismen von Graphen''. The article is a part of the first author's PhD thesis.}
\thanks{We thank Adam Skalski for helpful discussions regarding the anticommutative twists.}
\thanks{We thank the referee for helpful comments, which substantially improved the present article.}

\begin{abstract}
We present an algorithm for approximating linear categories of partitions (of sets). We report on concrete computer experiments based on this algorithm which we used to obtain first examples of so-called non-easy linear categories of partitions. All of the examples that we constructed are proven to be indeed new and non-easy. We interpret some of the new categories in terms of quantum group anticommutative twists.
\end{abstract}

\maketitle
\section*{Introduction}
By a \emph{partition} we mean a partition of a set, that is, a decomposition of a given finite set into disjoint non-empty subsets (see e.g. \cite{Sta11}). On the set of all partitions one can define a linear structure and operations of composition, tensor product, and involution giving it the structure of a monoidal $*$-category. By a \emph{partition category} we mean any subcategory of this one.

Partition categories, also known as \emph{(linear) categories of partitions}, have been heavily studied by researchers from different fields of mathematics and physics such as group theory \cite{Bra37,Wen88,HR05}, compact quantum groups \cite{BS09,Web13,RW16,TW18}, operator algebras \cite{Web17Cstar}, tensor categories \cite{Del07,CO11,CH17} or statistical physics \cite{TL71,Kau87,Mar94}.

Our motivation for studying those structures comes from the theory of (compact quantum) groups \cite{Wor87,NT13}, where those categories model the representation theory of a given quantum group. Our goal is to construct new kinds of examples of partition categories since those induce examples of compact matrix quantum groups (see the so-called ``easy'' \cite{BS09,Web17} and ``non-easy'' \cite{Maa18,GW18} quantum groups). We are particularly interested in concrete examples of ``non-easy'' categories and associated quantum groups~-- a class on which basically nothing was known until recently. A linear category of partitions is called {\em non-easy} whenever working with non-trivial linear combinations of partitions is essential to describe it.

So, the main goal of our project lies in finding and analysing the first examples of non-easy linear categories of partitions as no examples were known before we started our work. The contents and main results of this paper can be divided into three parts. In the first part, we describe some computer experiments that lead to discovering new examples of non-easy categories. We implemented a simple algorithm that takes as an input a set of generators in the form of a linear combination of partitions and approximates the partition category it generates. We describe the idea in Section~\ref{sec.algorithm} and provide the concrete computations in Section~\ref{sec.results}.

Secondly, we study the categories by theoretical means and prove that they are indeed new and non-easy. This is done in Section~\ref{sec.noneasydirect}. The last part of this paper -- Section~\ref{sec.twists} -- is devoted to interpreting the new categories within the theory of compact matrix quantum groups. Most of the categories discovered here were actually studied from the quantum group perspective in a separate paper \cite{GW18}. In this article, we study some remaining cases, which can be interpreted in terms of some non-commutative twists of the orthogonal group.

\begin{table}[b]
\tabskip=0pt
\def\bstrut{\vrule height0pt depth4.5pt width0pt}
\def\tstrut{\vrule height10.5pt depth0pt width0pt}
\newdimen\tboxw
\def\textbox#1{\vcenter{\parindent=0pt\leftskip=0pt plus 1fill\rightskip=\leftskip\emergencystretch=2em\hsize=\tboxw\tstrut#1\bstrut}}
\noindent\halign to \hsize{
\tabskip=0pt plus1fill
\hfil\vrule height12pt depth3pt width0pt #\hfil &
$\vcenter{\hsize=0.4\hsize\parindent=0pt\def\crr{\hfill$\par$\hfill\scriptstyle}\tstrut$\hfil\scriptstyle #\hfil$\bstrut}$ &
\hfil\vrule height12pt depth3pt width0pt #\hfil &
\tabskip=0pt\hfil$#$\hfil\cr
\noalign{\hrule}
\omit\tboxw=5em$\textbox{Paragraph w/~candid.}$&
\omit\tboxw=0.35\hsize$\textbox{Generator as a~full linear combination}$&
\omit\tboxw=7em$\textbox{Section where it is studied}$&
\omit\tboxw=6em$\textbox{Systematical description}$\cr
\noalign{\hrule}
\ref{Cand1} &\delta^2\,\Laaa-\delta(\Labb+\Laab+\Laba)+2\Labc & \ref{secc.gencont}, \ref{secc.P} & \Pc_{(\delta)}\Laaa\cr
\ref{Cand1} &\left(-2(1+\delta)\mp(2+\delta)\sqrt{\delta+1}\right)\,\Laaa-\crr(1\pm\sqrt{\delta+1})(\Labb+\Laab+\Laba)+\Labc & \ref{secc.V} & \Vc_{(\delta,\pm)}\Laaa\cr
\ref{Cand2} &\delta^3\Laaaa-2\delta^2(\Laaab+\Laaba+\dots)+\crr4\delta(\Laabc+\Labac+\dots)-16\Labcd & \ref{secc.tau} & \Tc_{(\delta)}\Laaaa\cr
\ref{Cand2} &\delta^3(\delta+1)\Laaaa-\crr\delta^2(\delta+1\pm\sqrt{\delta+1})(\Laaab+\Laaba+\dots)+\crr\delta(\delta+2\pm2\sqrt{\delta+1})(\Laabc+\Labac+\dots)+\crr(\delta^2-4\delta-8\mp8\sqrt{\delta+1})\Labcd & \ref{secc.V} & \Vc_{(\delta,\pm)}\Laaaa\cr
\ref{Cand3} &\delta^2\Labab-2\delta(\Labac+\Labcb)+4\Labcd & \ref{secc.disjoin} & \Dc\Labab\cr
\ref{Cand3} &\Labab-2\Laaaa & \ref{secc.join} & \Jc\Labab\cr
\ref{Cand4} &\Laaaa-{1\over\delta}(\Laaab+\Labbb+\Labaa+\Laaba)+\crr{1\over\delta^2}(\Labac+\Labcb) & \ref{secc.P} & \Pc_{(\delta)}\Laaaa\cr\noalign{\hrule}
}\medskip
\caption{Summary of all generators of non-easy linear categories of partitions studied in this article}
\label{t.noneasygens}
\end{table}

To summarize the results of Section~\ref{sec.results} and give an overview of Section~\ref{sec.noneasydirect}, we list in Table~\ref{t.noneasygens} all the linear combinations of partitions appearing in this article that generate non-easy categories. In the first column, we give a reference to the corresponding paragraph in Section~\ref{sec.results}, where the linear combination was discovered. In the second column, we explicitly write down the linear combination of partitions. In the third column, we refer to the corresponding subsection of Sect.~\ref{sec.noneasydirect}, where the linear combination is studied. We interpret those linear combinations usually as images of some mappings and we give this interpretation in the last column. Note that the expressions in the second and the last column may not be exactly equal; however, they generate the same category.

Note also that Table~\ref{t.noneasygens} does not yet provide an exhaustive summary of all non-easy categories we found. We list here only the generators. Those generators can be further combined with other partitions to define additional non-easy categories.

\section{Preliminaries}


\subsection{Partitions}

Let $k,l\in\N_0$, by a \emph{partition} of $k$ upper and $l$ lower points we mean a partition of the set $\{1,\dots,k\}\sqcup\{1,\dots,l\}\approx\{1,\dots,k+l\}$, that is, a decomposition of the set of $k+l$ points into non-empty disjoint subsets, called \emph{blocks}. The first $k$ points are called \emph{upper} and the last $l$ points are called \emph{lower}. The set of all partitions on $k$ upper and $l$ lower points is denoted $\Part(k,l)$. We denote their union by $\Part:=\bigcup_{k,l\in\N_0}\Part(k,l)$. The number $\left| p\right|:=k+l$ for $p\in\Part(k,l)$ is called the \emph{length} of $p$.

We illustrate partitions graphically by putting $k$ points in one row and $l$ points on another row below and connecting by lines those points that are grouped in one block. All lines are drawn between those two rows.

Below, we give an example of two partitions $p\in \Part(3,4)$ and $q\in\Part(4,4)$ defined by their graphical representation. The first set of points is decomposed into three blocks, whereas the second one is into five blocks. In addition, the first one is an example of a \emph{non-crossing} partition, i.e.\ a partition that can be drawn in a way that lines connecting different blocks do not intersect (following the rule that all lines are between the two rows of points). On the other hand, the second partition has one crossing.
\begin{equation}
\label{eq.pq}
p=
\BigPartition{
\Pblock 0 to 0.25:2,3
\Pblock 1 to 0.75:1,2,3
\Psingletons 0 to 0.25:1,4
\Pline (2.5,0.25) (2.5,0.75)
}
\qquad
q=
\BigPartition{
\Psingletons 0 to 0.25:1,4
\Psingletons 1 to 0.75:1,4
\Pline (2,0) (3,1)
\Pline (3,0) (2,1)
\Pline (2.75,0.25) (4,0.25)
}
\end{equation}

A block containing a single point is called a \emph{singleton}. In particular, the partitions containing only one point are called singletons and for clarity denoted by an arrow $\singleton\in\Part(0,1)$ and $\upsingleton\in\Part(1,0)$. For more information about partitions, see \cite{Sta11,NS06,Web17}.

\subsection{Operations on partitions}
\label{secc.op}

Let us fix a complex number $\delta\in\C$. Let us denote $\Partlin_\delta(k,l)$ the vector space of formal linear combination of partitions $p\in\Part(k,l)$. That is, $\Partlin_\delta(k,l)$ is a vector space, whose basis is $\Part(k,l)$.

Now, we are going to define some operations on $\Partlin_\delta$. First, let us define those operations just on partitions.
\begin{itemize}
\item  The \emph{tensor product} of two partitions $p\in\Part(k,l)$ and $q\in\Part(k',l')$ is the partition $p\otimes q\in \Part(k+k',l+l')$ obtained by writing the graphical representations of $p$ and $q$ ``side by side''.
$$
\BigPartition{
\Pblock 0 to 0.25:2,3
\Pblock 1 to 0.75:1,2,3
\Psingletons 0 to 0.25:1,4
\Pline (2.5,0.25) (2.5,0.75)
}
\otimes
\BigPartition{
\Psingletons 0 to 0.25:1,4
\Psingletons 1 to 0.75:1,4
\Pline (2,0) (3,1)
\Pline (3,0) (2,1)
\Pline (2.75,0.25) (4,0.25)
}
=
\BigPartition{
\Pblock 0 to 0.25:2,3
\Pblock 1 to 0.75:1,2,3
\Psingletons 0 to 0.25:1,4,5,8
\Psingletons 1 to 0.75:5,8
\Pline (2.5,0.25) (2.5,0.75)
\Pline (6,0) (7,1)
\Pline (7,0) (6,1)
\Pline (6.75,0.25) (8,0.25)
}
$$

\item For $p\in\Part(k,l)$, $q\in\Part(l,m)$ we define their \emph{composition} $qp\in\Partlin_\delta(k,m)$ by putting the graphical representation of $q$ below $p$ identifying the lower row of $p$ with the upper row of $q$. The upper row of $p$ now represents the upper row of the composition and the lower row of $q$ represents the lower row of the composition. Each extra loop that appears in the middle and is not connected to any of the upper or the lower points, transforms to a multiplicative factor $\delta$.
$$
\BigPartition{
\Psingletons 0 to 0.25:1,4
\Psingletons 1 to 0.75:1,4
\Pline (2,0) (3,1)
\Pline (3,0) (2,1)
\Pline (2.75,0.25) (4,0.25)
}
\cdot
\BigPartition{
\Pblock 0 to 0.25:2,3
\Pblock 1 to 0.75:1,2,3
\Psingletons 0 to 0.25:1,4
\Pline (2.5,0.25) (2.5,0.75)
}
=
\BigPartition{
\Pblock 0.5 to 0.75:2,3
\Pblock 1.5 to 1.25:1,2,3
\Psingletons  0.5 to  0.75:1,4
\Pline (2.5,0.75) (2.5,1.25)
\Psingletons -0.5 to -0.25:1,4
\Psingletons  0.5 to  0.25:1,4
\Pline (2,-0.5) (3,0.5)
\Pline (3,-0.5) (2,0.5)
\Pline (2.75,-0.25) (4,-0.25)
}
= \delta^2
\BigPartition{
\Pblock 0 to 0.25:2,3,4
\Pblock 1 to 0.75:1,2,3
\Psingletons 0 to 0.25:1
\Pline (2.5,0.25) (2.5,0.75)
}
$$

\item For $p\in\Part(k,l)$ we define its \emph{involution} $p^*\in\Part(l,k)$ by reversing its graphical representation with respect to the horizontal axis.
$$
\left(
\BigPartition{
\Pblock 0 to 0.25:2,3
\Pblock 1 to 0.75:1,2,3
\Psingletons 0 to 0.25:1,4
\Pline (2.5,0.25) (2.5,0.75)
}
\right)^*
=
\BigPartition{
\Pblock 1 to 0.75:2,3
\Pblock 0 to 0.25:1,2,3
\Psingletons 1 to 0.75:1,4
\Pline (2.5,0.25) (2.5,0.75)
}
$$
\end{itemize}

Now we can extend the definition of tensor product and composition on the vector spaces $\Partlin_\delta(k,l)$ linearly. We extend the definition of the involution antilinearly. These operations are called the \emph{category operations} on partitions. See \cite{TW18} for more examples of the category operations.

\subsection{Linear categories of partitions}
\label{secc.cat}
The set of all natural numbers with zero $\N_0$ as a set of objects together with the spaces of linear combinations of partitions $\Partlin_\delta(k,l)$ as sets of morphisms between $k\in\N_0$ and $l\in\N_0$ with respect to those operations form a monoidal $*$-category. All objects in the category are self-dual.

We are interested in subcategories of the category of all partitions. That is, any collection $\Kat$ of linear subspaces $\Kat(k,l)\subset\Partlin_\delta(k,l)$ containing the \emph{identity partition} $\idpart\in\Kat(1,1)$ and the \emph{pair partition} $\pairpart\in\Kat(0,2)$, which is closed under the category operations is called a \emph{linear category of partitions}.

For given $p_1,\dots,p_n\in\Partlin_\delta$, we denote by $\langle p_1,\dots,p_n\rangle_\delta$ the smallest linear category of partitions containing $p_1,\dots,p_n$. We say that $p_1,\dots,p_n$ \emph{generate} $\langle p_1,\dots,p_n\rangle_\delta$. Note that the identity partition and the pair partition are contained in the category by definition and hence will not be explicitly listed as generators. Any element in $\langle p_1,\dots,p_n\rangle_\delta$ can be obtained from the generators $p_1,\dots,p_n$, the identity partition $\idpart$, and the pair partition $\pairpart$ by performing a finite amount of category operations and linear combinations. 

Instead of having different categories for different parameters $\delta$, we can consider ``all of them at once''. That is, define a category $\Partlin$, where the morphism spaces $\Partlin(k,l)$ are modules over the polynomial ring $R:=\C[\delta]$.

\subsection{Partitions with points on one line}
In order to describe elements of a given linear category of partitions $\Kat$, we actually do not have to describe all the spaces $\Kat(k,l)$ for all $k,l\in\N_0$. For $p\in\Part(k,l)$, $k>0$, its \emph{left rotation} is a partition $\Lrot p\in\Part(k-1,l+1)$ obtained by moving the leftmost point of the upper row on the beginning of the lower row. Similarly, for $p\in\Part(k,l)$, $l>0$, we can define its \emph{right rotation} $\Rrot p\in\Part(k+1,l-1)$ by moving the last point of the lower row to the end of the upper row. Both operations are obviously invertible. We extend this operation linearly on $\Partlin_\delta$.

\begin{prop}[{\cite[Lemma 2.7]{BS09}}]
Every category $\Kat$ is closed under left and right rotations and their inverses.
\end{prop}

This means that every category $\Kat$ is described by the spaces $\Kat(0,l)$ with lower points only since the spaces $\Kat(k,l)$ can be obtained by rotation of $\Kat(0,k+l)$.

\subsection{Representing partitions by words}
For partitions on one line, that is, with lower points only, we can define an alternative way of representing them. Instead of pictures, we can use words. Given a partition $p\in\Part(0,l)$, we can mark its blocks by letters and represent $p$ as a word $a_1\cdots a_l$, where $a_i$ is a letter corresponding to the block of the $i$-th point. For example, rotating the partitions $p$ and $q$ from Equation \eqref{eq.pq}, we obtain
$$
p'=\LPartition{0.4:4,7}{0.8:1,2,3,5,6},\qquad
q'=\LPartition{0.4:1,4,5}{0.8:3,7,8;1.2:2,6}.
$$
Those can be represented by words as
$$p'=\mathsf{aaabaac},\qquad q'=\mathsf{abcdebcc}.$$
Note that the word representation is not unique. Choosing different set of letters, we can also write for example $p'=\mathsf{dddaddf}$ or $q'=\mathsf{defgheff}$.

Now, let us define some operations on linear combinations of partitions with lower points only and express them in terms of words.
\begin{itemize}
\item The \emph{tensor product} of two partitions with lower points only is again a partition with lower points only. Let $p_1$ be represented by a word $w_1$ and $p_2$ by a word $w_2$ such that $w_1$ and $w_2$ contain disjoint sets of letters. Then $p_1\otimes p_2$ is represented by the word $w_1w_2$. For the example above,
$$\mathsf{aaabaac}\otimes\mathsf{abcdebcc}=\mathsf{aaabaac}\otimes\mathsf{defgheff}=\mathsf{aaabaacdefgheff}.$$
\item For $p\in\Partlin_\delta(0,l)$, $l\ge 2$, we define its \emph{contraction} as $\Pi p:=qp$, where $q=\uppairpart\otimes\idpart^{\otimes l-2}$. On the basis of partitions it can be described using word representation by identifying the first two letters and then removing them. If the first two letters are the only occurrence of those letters in the word, we multiply by a factor $\delta$.
$$\Pi(\mathsf{abcadbc})=\mathsf{cadac},\qquad\Pi(\mathsf{aabcdc})=\delta\,\mathsf{bcdc}.$$
\item For $p\in\Partlin_\delta(0,l)$ we define its \emph{rotation} $Rp:=(\Lrot\circ\Rrot)(p)$. For a partition in word representation this operation takes the last letter and puts it in front of the word.
$$R(\mathsf{abcdebcc})=\mathsf{cabcdebc}.$$
\item For $p\in\Partlin_\delta(0,l)$ we define its \emph{reflection} $p^\star:=\Lrot^l(p^*)$. For a partition in word representation this operation reverses the order of the letters.
$$(\mathsf{abcdebcc})^\star=\mathsf{ccbedcba}.$$
\end{itemize}

The above defined operations on partitions on one line will be called the \emph{word operations}. They were defined using the category operations of tensor product, composition, and involution. In the following proposition we prove that conversely, the category operations can be expressed in terms of the word operations. One could say that category operations and word operations are in this sense equivalent.

\begin{prop}
For any linear category of partitions $\Kat$, the collection of spaces $\Kat(0,l)$, $l\in\N_0$ is closed under the word operations. Conversely, for any collection of linear subspaces $\Kat(0,l)\subset\Partlin_\delta(0,l)$ closed under the word operations, the collection
$$\Kat(k,l):=\{\Rrot^kp\mid p\in\Kat(0,k+l)\}=\{\Lrot^{-k}p\mid p\in\Kat(0,k+l)\}$$
is closed under the category operations, so it is a linear category of partitions.
\end{prop}
\begin{proof}
The word operations are defined using the category operations. From this, the first part of the proposition follows.

The second part is proven by expressing the category operations using the word operations
$$\Lrot^{-k}p\otimes\Rrot^{k'}q=\Lrot^{-k}\Rrot^{k'}(p\otimes q),$$
$$(\Rrot^k p)^*=\Rrot^l p^\star,$$
$$(\Rrot^l q)(\Rrot^k p)=\Rrot^k\Pi_{m+1}\Pi_{m+2}\cdots\Pi_{m+l}(q\otimes p).$$
In the last row, we assume that $p\in\Part(0,k+l)$ and $q\in\Part(0,l+m)$ and we denote $\Pi_i:=R^i\circ\Pi\circ R^{-i}$.
\end{proof}

This allows us to work just with partitions with lower points only.


\section{The problem and its motivation}

The motivation for our computation is the following. Suppose the parameter $\delta$ is actually a natural number $N\in\N$. Then one can define \cite{BS09} a functor $T$ from the category $\Partlin_N$ to the category of matrices $\Mat$ mapping elements $p\in\Partlin_N(k,l)$ to matrices representing linear maps $T_p\colon (\C^N)^{\otimes k}\to (\C^N)^{\otimes l}$.

Given a (quantum) group $G$ and its representations $\phi$, $\psi$, we denote by $\Mor(\phi,\psi)$ the linear space of all intertwiners between $\phi$ and $\psi$. It holds that representations of a given (quantum) group again form a category and this category can be modelled using partition categories via the functor $T$.

Let us consider, for example, the orthogonal group $O_N$ and its fundamental representation $\phi\colon O_N\to\GL(N,\C)$. Consider also the linear category of all pair partitions $\Pair_N\subset\Partlin_N$ (partitions, where all blocks are of size two). Then by Brauer's gerneralization of the Schur--Weyl duality \cite{Bra37}, we have that
$$\Mor(\phi^{\otimes k},\phi^{\otimes l})=\{T_p\mid p\in\Pair_N(k,l)\}.$$
Similarly, considering the symmetric group $S_N$ and its representation $\psi\colon S_N\to\GL(N,\C)$ by permutation matrices, we can model the intertwiners using the category of all partitions \cite{HR05,BS09}
$$\Mor(\psi^{\otimes k},\psi^{\otimes l})=\{T_p\mid p\in\Partlin_N(k,l)\}.$$

By so-called Tannaka--Krein duality, we have also the converse direction: Every compact group can be recovered from its representation theory. Thus, there is a mutual correspondence between categories $\Kat$ with 
$$\Pair_N\subset\Kat\subset\Partlin_N$$
and compact groups $G$ with
$$O_N\supset G\supset S_N.$$

The Tannaka--Krein duality was generalized by Woronowicz to the case of so-called compact quantum groups in \cite{Wor88} (the definition of compact quantum groups is also due to Woronowicz \cite{Wor87}). This also generalizes this correspondence. The smallest category of partitions $\NCPair$ consisting of all non-crossing pair partitions corresponds to the so-called \emph{free orthogonal quantum group} $O_N^+$ defined originally by Wang \cite{Wan95free}. Thus, we have a correspondence between all categories of partitions
$$\NCPair_N\subset\Kat\subset\Partlin_N$$
and compact quantum groups $G$ with 
$$O_N^+\supset G\supset S_N.$$
We are interested in finding examples (and possibly a classification) of partition categories.

A lot of success was achieved using a great idea of Banica and Speicher \cite{BS09} to consider categories, where the morphism spaces $\Kat(k,l)$ have a basis in terms of partitions (in contrast with a basis given only in terms of linear combinations of partitions). This allows to completely ignore the linear structure of the partition category and the problem becomes purely combinatorial. Such categories are called \emph{easy} and they were studied in many articles and finally their complete classification was found in \cite{RW16}.

\begin{defn}
\label{D.easy}
A linear category of partitions $\Kat$ is called \emph{easy} if, for every $k,l\in\N_0$, there is a set $\Cat(k,l)\subset\Part(k,l)$ such that $\Kat(k,l)=\spanlin\Cat(k,l)$. Otherwise it is called \emph{non-easy}.
\end{defn}

\begin{rem}
\label{R.easy}
The category operations map partitions to scalar multiples of partitions. This implies that any category that is generated by partitions is surely easy. (Iterating the category operations on generators we never get a linear combination of more than one partition.)
\end{rem}

In the non-easy case, very few results are available. This motivated the authors to use the computer to look for some examples. So, since the easy case is classified, we can focus on a more specific problem:

\begin{prob}
\label{prob}
Find examples of non-easy partition categories.
\end{prob}

Finally, let us mention a few remarks on the current status of theoretical research in this area. As we already mentioned, there were no non-easy categories known before we started our project. However, this has changed in the meantime. A lot of examples of non-easy quantum groups was recently discovered by Maassen by studying so-called group-theoretical quantum groups \cite{Maa18}. Another recent work dealing with non-easy categories is by Banica \cite{Ban18}. We would also like to mention another very active research direction, namely generalizing partition categories by colouring points \cite{Fre17}. The most interesting case are probably the two-coloured partitions describing unitary quantum groups. In this case, the classification is not complete even in the easy case. The known classification results are \cite{TW18,Gro18,MW18,MW19,MW19b,MW20}.

\section{The algorithm}
\label{sec.algorithm}

The idea of using a computer to find examples of non-easy categories is very simple. Consider a linear combination of partitions $p\in\Partlin_\delta(k,l)$ and try to generate the whole category $\Kat:=\langle p\rangle_\delta$ by iterating the category operations on the set $\{\idpart,\pairpart,p\}$. Unfortunately, there is no theoretical result that would assure that, after performing a given amount of category operations on the generators, we get all elements of $\Kat(i,j)$ for some $i,j\in\N_0$. That is, we are not aware of any algorithmic approach that could prove non-easiness of a category. However, we are able to prove \emph{easiness} of a category and hence, excluding the easy cases, obtain at least candidates for the non-easy categories. The precise way, how we use this to look for non-easy categories is described in Section \ref{sec.results}.

\subsection{Some observations}
Let us mention some observations making our computation easier.

First, when looking for examples of non-easy categories, it makes sense to look just for the categories generated by one element.
\begin{prop}
Let $p_1,\dots,p_n\in\Partlin_\delta(k,l)$. If $\langle p_1,\dots,p_n\rangle_\delta$ is non-easy, then at least one of the categories $\langle p_1\rangle_\delta,\dots,\langle p_n\rangle_\delta$ is non-easy.
\end{prop}
\begin{proof}
If all the categories $\langle p_i\rangle_\delta$ are easy, then $\langle p_1,\dots, p_n\rangle_\delta$ is generated by partitions, which, according to Remark \ref{R.easy}, implies that it is easy.
\end{proof}

Secondly, the following proposition describes how to prove easiness of the category~$\langle p\rangle_\delta$.

\begin{prop}
\label{P.summands}
Consider $p\in\Partlin_\delta(k,l)$ and express it in the basis of partitions as $p=\sum_{i}\alpha_ip_i$, where $\alpha_i\in\C$ are non-zero numbers and $p_i\in\Part(k,l)$ are mutually different partitions. Then the category $\langle p\rangle_\delta$ is easy if and only if it contains all the partitions $p_i$.
\end{prop}
\begin{proof}
Left-right implication follows from uniqueness of coordinates with respect to a given basis. Right-left multiplication follows from Remark \ref{R.easy}.
\end{proof}

Thirdly, the following result further reduces the computational complexity. In particular, it allows to avoid using the antilinear operation of reflection.

\begin{prop}
\label{P.refl}
Let $S$ be a set of linear combinations of partitions on one line which is closed under the operation of reflection and contains the pair partition $\pairpart$. Then any element of $\langle S\rangle$ can be obtained by performing a finite amount of tensor products, contractions and rotations and taking linear combinations. It is automatically closed under reflections.
\end{prop}
\begin{proof}
We have
\begin{align*}
(p\otimes q)^\star&=q^\star\otimes p^\star,\\
(\Pi\,p)^\star&=(R^{-2}\circ\Pi\circ R^2)p^\star,\\
(Rp)^\star&=R^{-1}p^\star.\qedhere
\end{align*}
\end{proof}

Finally, the following proposition allows to reduce the amount of generators $p$ we have to consider.

\begin{prop}
\label{P.gcd}
Consider $p,q\in\Partlin_\delta(0,k)$ and let $f$ be a polynomial of degree less than $k$. Then $\langle f(R)p+q\rangle=\langle g(R)p+\tilde q\rangle$, where $g(x)=\gcd(f(x),x^k-1)$ and $\tilde q\in\Partlin_\delta(0,k)$ is a linear combination of rotations of $q$.
\end{prop}
\begin{proof}
Consider $f(x)$ as an element of the algebra $A:=\C[x]/I$, where $I$ is the ideal generated by $x^k-1$. Since $R^k=I$, the evaluation $h(R)$ for $h\in A$ does not depend on the particular representative.

There certainly exists $\tilde f\in A$ such that $f=\tilde f g$ and $\tilde f$ is coprime to $x^k-1$ (just take $\tilde f(x)=(f(x)+j(x^k-1))/g(x)$ for appropriate $j\in\N$). Then, $\tilde f$ as an element of $A$, is not a divisor of zero and hence, since $A$ is finite dimensional, $\tilde f$ is invertible. Therefore, there exists $h\in A$ such that $hf=h\tilde fg=g$ and we have that $\langle f(R)p+q\rangle\supset\langle h(R)(f(R)p+q)\rangle=\langle g(R)p+\tilde q\rangle$, where $\tilde q:=h(R)q$.

The opposite inclusion is easy $\langle g(R)p+\tilde q\rangle\supset\langle \tilde f(R)(g(R)p+\tilde q)\rangle=\langle f(R)p+q\rangle$.
\end{proof}

%

\subsection{Preprocesing}
First, we need to compute the matrices of the operations of tensor product, contraction and rotation as (bi)linear maps. Note that the number of partitions of $l$ points is given by \emph{Bell numbers} $B_l$. So, the dimension of $\Partlin_\delta(0,l)$ is $B_l$. Thus, we can identify $\Partlin_\delta(0,l)\simeq\C^{B_l}$ identifying the partitions $p\in\Part(0,l)$ with the standard basis in $\C^{B_l}$. Actually, it is more convenient not to specify the value of $\delta$ and consider rather $\Partlin(0,l)\simeq R^{B_l}$ for $R:=\C[\delta]$.

The tensor product $\otimes\colon\Partlin(0,k)\times\Partlin(0,l)\to\Partlin(0,k+l)$ of partitions can be viewed as a linear map
$$\tens\colon R^{B_kB_l}\to R^{B_{k+l}}.$$
Similarly, we can define the matrices corresponding to contraction and rotation
$$\contr\colon R^{B_l}\to R^{B_{l-2}},\qquad\rot\colon R^{B_l}\to R^{B_l}.$$

We fix a \emph{length bound} $l_0\in\N_0$ and compute all those matrices for $l\le l_0$ (resp. $k+l\le l_0$ in case of the tensor product).

\subsection{Adding procedures}
We define modules $K_l\subset R^{B_l}$ for $l\le l_0$ that correspond to the spaces $\Kat(0,l)$ of some category $\Kat\subset\Partlin$. We define the following procedures.

The procedure \textsc{AddParts} takes a set $S\in R^{B_l}$ representing a set of linear combinations of partitions from $\Kat(0,l)$ and adds it to the module $K_l$. In addition, it adds all the rotations of the partitions to $K_l$ and all their contractions to the corresponding $K_{l-2i}$. Thus, we end up with an approximation of $\Kat$, which contains the set $S$ and is closed under taking rotations and contractions. 

\begin{algorithm}
\caption{Adding a set of partitions to $K_l$}\label{addm}
\begin{algorithmic}[1]
\Procedure{AddParts}{$l\in\{1,\dots,l_0\},S\subset R^{B_l}$}
\If{$l\ge 2$}
	\State \textsc{AddParts}($l-2,\contr(S)$)
\EndIf
\State $K_l:=K_l+S$
\For{$j\in\{1,\dots,l-1\}$}
	\State $S:=\rot(S)$
	\If{$l\ge 2$}
		\State \textsc{AddParts}($l-2,\contr(S)$)
	\EndIf
	\State $K_l:=K_l+S$
\EndFor
\EndProcedure
\end{algorithmic}
\end{algorithm}

The procedure \textsc{AddTensors} takes all pairs $x\in K_k$ and $y\in K_l$ such that $k+l\le l_0$ and computes the vector corresponding to the partition tensor product $\tens(x\otimes y)$. Note that we can assume $k\le l$ since we have $q\otimes p=R^l(p\otimes q)$ for $p\in\Partlin(0,k)$ and $q\in\Partlin(0,l)$. To add the results to the category approximation, we use the procedure \textsc{AddParts}, so we add also all the rotations and contractions of the tensor products.

\begin{algorithm}
\caption{Add tensor products to $K_i$'s}\label{addcat}
\begin{algorithmic}[1]
\Procedure{AddTensors}{}
\For{$k\in\{1,\dots,\lfloor l_0/2\rfloor\}$}
	\For{$l\in\{k,\dots,l_0-k\}$}
		\State\textsc{AddParts}($k+l,\tens(K_k\otimes K_l)$)
	\EndFor
\EndFor
\EndProcedure
\end{algorithmic}
\end{algorithm}

\begin{table}[b]
\begin{tabular}{r|cccccccccccc}
  $l$  &  1 & 2 & 3 &  4 &  5 &   6 &   7 &    8 &     9 &     10 &     11 &      12 \\\hline
$B_l$  &  1 & 2 & 5 & 15 & 52 & 203 & 877 & 4\,140 & 21\,147 & 115\,975 & 678\,570 & 4\,213\,597
\end{tabular}
\caption{Bell numbers}
\label{t.bell}
\end{table}

\subsection{The algorithm}
Suppose for simplicity, we have one generator $p\in\Partlin(0,l_1)$, $l_1\le l_0$. Then we can compute an approximation of $\Kat:=\langle p\rangle$ by performing the following algorithm.
\begin{enumerate}
\item \textsc{AddParts}($2,\pairpart$); \textsc{AddParts}($l_1,\{p,\tilde p\}$);
\item Repeat \textsc{AddTensors}() until this procedure leaves all the modules $K_l$ unchanged.
\end{enumerate}
At this stage, our category approximation is closed under contractions, rotations, reflections, and tensor products whose result has length lower or equal to the length bound $l_0$. (Note that the closedness with respect to reflections follows from Proposition \ref{P.refl}.)

\subsection{Limits of the algorithm}
The fact that the category approximation is closed under the category operations in the above sense, however, does not mean that our approximation is faithful since it may happen that in order to obtain a partition on $l$ points for $l\le l_0$ we need to compute an intermediate result with length greater than the length bound $l_0$ first.

If we need more reliable approximation, we need to increase the length bound $l_0$. Note that if we choose the length bound $l_0$ to be lower than $2l_1$ then our algorithm cannot even compute $p\otimes p$ for the generator $p$, so we can expect the results to be quite unreliable.

The value of the length bound $l_0$ has, of course, its limits. The Bell numbers $B_l$ grow exponentially with $l$, so the module dimensions become huge very quickly. In Table \ref{t.bell}, we list the Bell numbers for some small $l$. We see that the maximal value of $l_0$, which can be achieved for usual computer, is about $l_0=10$. In Section \ref{sec.results}, we will discuss results for generators of length $l_1\le 4$, which is pretty much close to the maximum that can be achieved without further assumptions.

Note that it may also be convenient to add some extra variables $a_1,\dots,a_m$ to the ring $R$. Then we can start with a generator $p$ depending on $a_1,\dots,a_m$ as parameters. However, each extra variable again notably increases the computation complexity. See Section \ref{sec.results} for how one can handle such extra parameters.

%
%
%
%

\section{Concrete computations}
\label{sec.results}

In this section, we are going to present concretely how our algorithm is applied. The algorithm was implemented in \textsc{Singular} \cite{Singular}. We also used Maple\footnote{Maple is a~trademark of Waterloo Maple Inc.} \cite{Maple} for solving systems of polynomial equations. In all computations that follow, we use the length bound $l_0:=8$.

Let us comment a bit on the nature of the results obtained in this section.

\begin{rem}[On the character of results in this section]
This section contains the main results of the computational part of the article. The results will typically have the following form
\begin{quote}
The following categories constitute the only possible candidates for non-easy categories of the form $\langle p\rangle_\delta$, where $\delta\in\C$ and $p\in\Part(k)$ satisfy \emph{those and those} assumptions\dots
\end{quote}
We use our algorithm to obtain such results. The code of our implementation is available on github\footnote{See \href{https://github.com/gromadan/partcat/}{github.com/gromadan/partcat}} so that everybody can check our computations. In principle, it is also possible to trace back the actual computations the computer did and prove those results by hand. We will indicate how to do this in the most simple non-trivial case of generator on three points in Section~\ref{secc.len3}.

Nevertheless, these statements and their assumptions are going to be rather technical. They should not be considered as some fundamental mathematical results of the theory and hence it would probably not be very useful to try to formulate detailed mathematical proofs for each of them. The statements are interesting mainly for the reason that they bring a list of candidates for non-easy categories, with which we can work afterwards. To emphasize this, the corresponding paragraphs formulating those statements will not be labelled as \emph{theorems} or \emph{propositions}, but simply as \emph{candidates}. The truly mathematical work then starts in Section~\ref{sec.noneasydirect}, where we prove that those candidates indeed constitute instances of new non-easy linear categories of partitions.
\end{rem}

\begin{rem}[On linear independence and semisimplicity]
We would like to stress that, by definition, we assume all partitions to be linearly independent. This might be a bit confusing for researchers coming from quantum groups as the linear maps $T_p$ associated to partitions are not linearly independent. Nevertheless, in this paper, we do not work with the maps $T_p$, but with the partitions and their formal linear combinations themselves.

The linear functor mapping $p\mapsto T_p$ is then surely non-injective. As a matter of fact, the kernel of this functor can be described in a very nice way: We can define so-called \emph{negligible morphisms} to be those $p\in\Partlin(0,k)$ such that $q^*p=0$ for every $q\in\Partlin(0,k)$. Those morphisms are then mapped to zero under arbitrary unitary fibre functor $p\mapsto T_p$. See e.g.\ \cite{FM20} for more information.

One might then ask, whether all the candidates produced in this section are so-to-say honestly distinct and non-easy or whether they differ only by some negligible morphisms. First, note that for a generic $\delta$, there are typically no negligible morphisms (for instance, in case of the category of all partitions $\Partlin_\delta$, there exist negligible morphisms only if $\delta\in\N_0$ \cite{Del07,FM20}). Secondly, if we are interested in applications in quantum groups, we choose $\delta=N\in\N$. In this case, there surely are some negligible morphisms $p\in\Partlin(0,k)$; however, they appear only for $k>N$ (see e.g.\ \cite[Cor.~3.4]{GW18}). So, considering two distinct partition categories generated by partitions on $k\le k_0$ points (in this section, we consider $k\le 4$), then they surely correspond to distinct quantum groups if we take $\delta=N>k_0$.
\end{rem}

\begin{rem}[On the choice of $\delta$]
During our computations we are going to make some restrictions on the parameter $\delta$. The reason is that for some concrete parameter values, there may exist some additional categories of partitions, which we do not consider to be interesting and hence we want to neglect of them. Due to our quantum group motivation and for reasons mentioned in the previous remark, we are particularly interested in categories that are non-easy for $\delta\in\N$, $\delta>4$.

Typically if $\delta$ is some small real number, it happens that our algorithm would find negligible morphisms as generators of non-easy categories. We are not interested in those as they do not define any new quantum groups. Secondly, we often exclude negative $\delta$ since again those categories do not have a straightforward interpretation in the theory of quantum groups. Nevertheless, the reader who is interested in the presented solutions as examples of abstract monoidal categories may want to go through the whole computation again and analyse the cases that we skipped here. In addition, we should mention that some solutions working for $\delta=4$ that we skipped here may be relevant also for the theory of compact quantum groups.
\end{rem}

\subsection{Generator of length one and two}

The space~$\Partlin(1)$ is one-dimensional being the span of the singleton partition. Therefore, any category generated by an element of length one is easy.

Similarly for the length two. We have $\Partlin(2)=\spanlin\{\pairpart,\Lab\}$. Since $\pairpart$ is in any category by definition, we again have that any category generated by an element of length two is easy.

\subsection{Generator of length three}
\label{secc.len3}

For $l=3$, we have the following partitions
$$\Part(3)=\{\Laaa,\Laab,\Laba,\Labb,\Labc\}.$$

So, a~general element $p\in\Partlin_\delta(3)$ can be expressed as follows
$$p=a\,\Laaa+b_1\,\Laab+b_2\,\Labb+b_3\,\Laba+c\,\Labc,$$
where $a,b_1,b_2,b_3,c\in\C$. Now, our goal is to exclude such values of those parameters, for which $\Kat:=\langle p\rangle_\delta$ is easy.

\begin{lem}
\label{L.3crit}
A~linear category $\Kat=\langle p\rangle_\delta$ with $p\in\Partlin_\delta(3)$ is easy if and only if $\singleton\in\Kat$. Hence, $\Kat$~is non-easy if and only if $\Kat(1)$ is empty.
\end{lem}
\begin{proof}
If $\singleton\in\Kat$, then all the partitions $\Laab$, $\Labb$, $\Laba$, and~$\Labc$ are in~$\Kat$. If $p$ contains also $\Laaa$ as a~summand, then also $\Laaa\in\Kat$. So, we have either $\Kat=\langle\singleton\rangle_\delta$ or $\Kat=\langle\Laaa,\singleton\rangle_\delta$. In both cases $\Kat$ is easy. Conversely, if $\Kat$ is easy, then it must contain at least one of the partitions in $\Part(3)$. Each of them generate the singleton.
\end{proof}

Running \textsc{AddParts}$(p)$ (over the ring $\C[\delta,a,b_1,b_2,b_3,c]$), we get immediately that $\Kat(1)$ contains the following elements\footnote{Details of the computation at \url{https://nbviewer.jupyter.org/github/gromadan/partcat/blob/master/lincat/Section_4_2.ipynb}}
$$(a+b_1+b_2+\delta b_3+\delta c)\singleton,\quad (a+b_1+\delta b_2+b_3+\delta c)\singleton,\quad (a+\delta b_1+b_2+b_3+\delta c)\singleton.$$
If $\Kat$ is non-easy, then $\Kat(1)$ must be empty, which leads to equations
\begin{align*}
a+\phantom\delta b_1+\phantom\delta b_2+\delta b_3+\delta c&=0\cr
a+\phantom\delta b_1+\delta b_2+\phantom\delta b_3+\delta c&=0\cr
a+\delta b_1+\phantom\delta b_2+\phantom\delta b_3+\delta c&=0.
\end{align*}
By subtracting the equations one from each other, we get $b_i(1-\delta)=b_j(1-\delta)$ for $i,j=1,2,3$. Suppose $\delta\neq 1$, then non-easiness implies that $b:=b_1=b_2=b_3$. Substituting this to one of the equations, we get an additional condition
$$a+(2+\delta)b+\delta c=0.$$
So, we can put $a:=-(2+\delta)b-\delta c$. Now, we can run our algorithm again over $\C[\delta,b,c]$. After one iteration of \textsc{AddTensors}, we get that $\Kat$ contains 
$$(\delta-1)(\delta-2)(\delta c+2b)(\delta c^2+2bc-b^2)\singleton.$$
Thus, excluding the case $\delta=1,2$, the category can be non-easy only if 
$$b=-c\delta/2\quad\hbox{or}\quad b=(1\pm\sqrt{\delta+1})c.$$
For $c=0$, we have also $b=0$, so the category is easy. For $c\neq 0$, we can normalize~$p$ dividing by~$c$.

\begin{cand}
\label{Cand1}
Assuming $\delta\in\C\setminus\{0,1,2\}$, the following are the only candidates on non-easy linear categories of partitions that are generated by a single element $p\in\Partlin_\delta(3)$:
\begin{gather*}
\langle \delta^2\,\Laaa-\delta(\Labb+\Laab+\Laba)+2\Labc\rangle_\delta,\cr
\left\langle \left(-2(1+\delta)-(2+\delta)\sqrt{\delta+1}\right)\,\Laaa-(1+\sqrt{\delta+1})(\Labb+\Laab+\Laba)+\Labc\right\rangle_\delta,\cr
\left\langle \left(-2(1+\delta)+(2+\delta)\sqrt{\delta+1}\right)\,\Laaa-(1-\sqrt{\delta+1})(\Labb+\Laab+\Laba)+\Labc\right\rangle_\delta.
\end{gather*}
\end{cand}

\begin{rem}
We could have derived the equations providing the conditions for non-easiness even without our algorithm. Indeed, the linear ones can be written as $\Pi R^ip=0$ for $i=0,1,2$ and the quadratic one as 
$$\Pi_2\Pi_3(p\otimes p)=\ethine(p\otimes p)=0.$$
This also provides a rigorous mathematical prove of the above statement.

The actual non-easiness of the categories will be proven in Section~\ref{sec.noneasydirect}. The first category is proven non-easy by Prop.~\ref{P.contractions} or, alternatively, Prop.~\ref{P.Pcnoneasy}; the other two by Prop.~\ref{P.Vcnoneasy}.

The algorithm was useful first for providing the idea to solve such equations and secondly for checking (although not proving) that the categories remain non-easy even after more iterations of the tensor product.
\end{rem}

\begin{rem}
The fact that there are (up to scaling) only three isolated candidates of non-easy generators already has interesting consequences. First of all, it follows that if those categories indeed are non-easy (which will be proven in Section~\ref{sec.noneasydirect}), then they must necessarily be distinct. Indeed, pick two of the categories and denote by $p_1$ and $p_2$ their generators. If we had $\langle p_1\rangle_\delta=\langle p_2\rangle_\delta$, then $p_1+tp_2$ would be a one-parameter family of linear combinations of partitions that generate a non-easy category.

Another consequence: The bottom two candidates cannot be non-easy unless $\delta>-1$ (that is, unless all the coefficients are real). The argumentation is similar. Denoting $p$ the corresponding generator, we necessarily need $p=p^\star$ since otherwise $p+tp^\star$ forms a one-parameter family of non-easy generators.
\end{rem}

\subsection{Generator of length four, case of no singletons}
A generator $p\in\Partlin_\delta(4)$ can be parametrized as follows
\begin{multline}
\label{eq.p1}
p=a_1\Laaaa+a_2\Labab+b_1\Laaab+b_2\Labbb+b_3\Labaa+b_4\Laaba+\\
c_1\Laabc+c_2\Labbc+c_3\Labcc+c_4\Labca+d_1\Labac+d_2\Labcb+e\Labcd.
\end{multline}
We omit the non-crossing pair partitions $\Laabb$ and~$\Labba$ since they are contained in every category.

Again, we want to exclude those parameters for which $\Kat:=\langle p\rangle_\delta$ is easy. Here, the situation is a~bit more complicated because we do not have a~criterion for easiness analogous to Lemma~\ref{L.3crit}. So, we divide the situation in different cases. In this section, we assume $\singleton\otimes\singleton\not\in\Kat$. We subdivide our computation even further:

\subsubsection{Generator not being rotationally symmetric}

First, let us briefly discuss the case, when $p$ is not rotationally symmetric. This means that 
\begin{multline*}
0\neq(R-1)p=:\tilde p=\tilde b_1\Laaab+\tilde b_2\Labbb+\tilde b_3\Labaa+\tilde b_4\Laaba+\cr
\tilde c_1\Laabc+\tilde c_2\Labbc+\tilde c_3\Labcc+\tilde c_4\Labca+d(\Labac-\Labcb),
\end{multline*}
where we denote $\tilde b_1=b_4-b_1$, $\tilde b_2=b_1-b_2$ and so on, so
\begin{align*}
\tilde b_1+\tilde b_2+\tilde b_3+\tilde b_4&=0,\\
\tilde c_1+\tilde c_2+\tilde c_3+\tilde c_4&=0.
\end{align*}

Denote by $\beta\colon\Partlin_\delta(2)\to\C$ the linear functional giving the coefficient of $\singleton\otimes\singleton$ for a~given linear combination $q\in\Partlin_\delta(2)$, i.e.\ mapping $\alpha\,\pairpart+\beta\,\singleton\otimes\singleton\mapsto\beta$. Since $\singleton\otimes\singleton\not\in\langle\tilde p\rangle_\delta$, we have four linear equations of the form $\beta(\Pi(R^i\tilde p))=0$, which read
\begin{align*}
\tilde b_1+\tilde b_4+\delta \tilde c_1+\tilde c_2+\tilde c_4&=0,\\
\tilde b_2+\tilde b_1+\delta \tilde c_2+\tilde c_3+\tilde c_1&=0,\\
\tilde b_3+\tilde b_2+\delta \tilde c_3+\tilde c_4+\tilde c_2&=0,\\
\tilde b_4+\tilde b_3+\delta \tilde c_4+\tilde c_1+\tilde c_3&=0.
\end{align*}
Together with the equations above, this leads to
$$\tilde c_3=-\tilde c_1,\quad \tilde c_4=-\tilde c_2,\quad \tilde b_2=-\tilde b_1-\delta \tilde c_2,\quad \tilde b_3=\tilde b_1+\delta(\tilde c_1+\tilde c_2),\quad \tilde b_4=-\tilde b_1-\delta \tilde c_1.$$

We can write $\tilde p=f(R)\Laabc+\tilde q$, where $f(x)=\tilde c_1+\tilde c_2x+\tilde c_3x^2+\tilde c_4x^3$. According to Proposition~\ref{P.gcd}, we can assume that $f$ is a~divisor of $x^4-1$. Thanks to the first two equations above, we see that $f(1)=0$ and $f(-1)=0$, so $f(x)$ is a~multiple of $x^2-1$. For $f(x)\neq 0$ (that is, either $f(x)=x^2-1$ or $f(x)=(x^2-1)\*(x\pm i)$), running one iteration of \textsc{AddTensors} shows that assuming $\delta\neq 2,4$ we have $\singleton\otimes\singleton\in\langle\tilde p\rangle_\delta$, which is a~contradiction.

In the case $f(x)=0$, we have
$$\tilde p=\tilde b(\Laaab-\Labbb+\Labaa-\Laaba)+\tilde d(\Labac-\Labcb).$$
One iteration of \textsc{AddTensors} yields $\tilde b=(-2\pm\sqrt{4-\delta})\tilde d$. Note that the involution acts on~$p$ by exchanging $\tilde b\mapsto -\bar{\tilde b}$ and $\tilde d\mapsto -\bar{\tilde d}$. Thus, both $\tilde b$ and~$\tilde d$ must be real up to scaling by a~complex number. This can be achieved only for $\delta\le 4$.

\begin{prop}
Consider $\delta\in\C\setminus(-\infty,4]$. Let $p\in\Partlin_\delta(4)$ such that $\singleton\otimes\singleton\not\in\Kat:=\langle p\rangle_\delta$ is non-easy. Then $p$ is rotationally symmetric.
\end{prop}
\begin{proof}
Follows from the considerations above.\footnote{Details of the computation at \url{https://nbviewer.jupyter.org/github/gromadan/partcat/blob/master/lincat/Section_4_3_1.ipynb}}
\end{proof}

\subsubsection{Rotationally symmetric generator}

Now, suppose $p$ is of the form
\begin{multline*}
p=a_1\Laaaa+a_2\Labab+b(\Laaab+\Labbb+\Labaa+\Laaba)+\\
c(\Laabc+\Labbc+\Labcc+\Labca)+d(\Labac+\Labcb)+e\Labcd.
\end{multline*}

Recall the notation $\beta\colon\Partlin_\delta(2)\to\C$ for the linear functional giving the coefficient of $\singleton\otimes\singleton$. As $\singleton\otimes\singleton\not\in\Kat$, we must have $\beta(q)=0$ for all $q\in\Kat(2)$. So, our idea for computing concrete coefficients providing a~candidate for a~non-easy category is to solve the following equations.
\begin{align}
\label{eq.phi1}\beta(\Pi p)&=0\\
\label{eq.phi2}\beta\left(\ethine(p\otimes p)\right)&=0\\
\label{eq.phi3}\beta\left(\cyclopropadiene(p\otimes p\otimes p)\right)&=0\\
\label{eq.phi4}\beta\left(\cyclobutatriene(p\otimes p\otimes p\otimes p)\right)&=0\\
\label{eq.phi5}\beta\left(\cyclopentaquartene(p\otimes p\otimes p\otimes p\otimes p)\right)&=0
\end{align}

\begin{rem}
\leavevmode
\begin{enumerate}
\renewcommand{\theenumi}{\alph{enumi}}
\item All the equations are homogeneous. (Their solution is obviously invariant with respect to scaling.)
\item The first equation containing one copy of~$p$ is linear, the second one is quadratic and so on.
\item The rotational symmetry reduces the number of variables and equations. Note for example that there is essentially just one way how to construct a~tensor product of two copies of~$p$ and then contract it to size two. Similarly for three copies of~$p$. For four copies, there are two additional ways, but it turns out that the corresponding equations already follow from Eqs. \eqref{eq.phi1}--\eqref{eq.phi4}.
\item The reflection acts on~$p$ by complex conjugating all the parameters. If it turns out that the system of equations has discrete solutions only (up to scaling), then the assumption of non-easiness implies that all the coefficients are up to scaling real. (Otherwise $p$ and~$p^\star$ are linearly independent, so $p+\alpha p^\star\in\langle p\rangle$ would be a~one-parameter set of solutions.)
\end{enumerate}
\end{rem}

We were not able to solve those equations in full generality. So, let us focus on some special cases.

\subsubsection{Special case: $a_2=0$, i.e.\ $p$ is non-crossing}

In this case, unless $b=c=d=e=0$, we have that $\singleton\otimes\singleton\not\in\langle p\rangle_\delta$ already implies that $\langle p\rangle_\delta$ is non-easy. Since we have only five variables, four homogeneous equations \eqref{eq.phi1}--\eqref{eq.phi4} are already enough to obtain a~list of discrete solutions (up to scaling). Using Maple, we found the following eight solutions:
\begin{equation}\label{eq.C2asol1}
a_1=1,\quad b=0,\quad c=0,\quad d=0,\quad e=0,
\end{equation}
\begin{equation}\label{eq.C2asol2}
a_1=\delta^3,\quad b=-2\delta^2,\quad c=4\delta,\quad d=4\delta,\quad e=-16,
\end{equation}
\begin{equation}\label{eq.C2asol3}
\begin{split}
a_1=\delta(\delta+1)(\delta+2\mp 2\sqrt{\delta+1}),\quad b&=\delta(-\delta-1\pm\sqrt{\delta+1}),\\
c=\delta,\quad d=\delta,\quad e&=\delta-2\mp2\sqrt{\delta+1},
\end{split}
\end{equation}
\begin{equation}\label{eq.C2asol4}
\begin{split}
a_1=2\delta^2(2\pm\sqrt{4-\delta}),\quad b=-\delta^2,\quad c=2\delta,\\
d=\mp\delta\sqrt{4-\delta},\quad e=2(-2\pm\sqrt{4-\delta}),
\end{split}
\end{equation}
\begin{equation}\label{eq.C2asol5}
\begin{split}
a_1=2\delta^3(3\mp 2\sqrt{3-\delta}),\quad b=\delta^2(-2\delta\pm\sqrt{3-\delta}),\quad c=\delta(4\delta-3),\\
d=\delta(\delta\pm 2(\delta-1)\sqrt{3-\delta}),\quad e=\pm2(3\delta-2)\sqrt{3-\delta}+7\delta-6.
\end{split}
\end{equation}

There are also some additional solutions for $\delta=-1$,~3,~4,~$(3\pm\sqrt{5})/2$, which we will not mention here. The first solution of the list above is the easy one. The following two lines~-- \eqref{eq.C2asol2} and~\eqref{eq.C2asol3} -- are interesting for us. Their non-easiness will be proven as Proposition~\ref{P.Tcnoneasy} and Proposition~\ref{P.Vcnoneasy}, respectively. Solutions \eqref{eq.C2asol4} and~\eqref{eq.C2asol5} are real only for $\delta\le 4$, resp. $\delta\le 3$, so we will ignore them here.

We can summarize the results in the following proposition.

\begin{cand}
\label{Cand2}
Consider $\delta\in\C\setminus(-\infty,4]$. Let $p\in\Partlin_\delta(4)$ be non-crossing such that $\Kat:=\langle p\rangle_\delta$ is non-easy and $\singleton\otimes\singleton\not\in\Kat$. Then $\Kat$ is equal to one of the following three categories
$$\langle\delta^3\Laaaa-2\delta^2(\Laaab+\Laaba+\dots)+4\delta(\Laabc+\Labac+\dots)-16\Labcd\rangle_\delta,$$
\begin{multline*}
\langle\delta^3(\delta+1)\Laaaa-\delta^2(\delta+1\pm\sqrt{\delta+1})(\Laaab+\Laaba+\dots)+\cr\delta(\delta+2\pm2\sqrt{\delta+1})(\Laabc+\Labac+\dots)+(\delta^2-4\delta-8\mp8\sqrt{\delta+1})\Labcd\rangle_\delta.
\end{multline*}
\end{cand}

Note that the two categories on the second line can be easy only for $\delta\in[-1,\infty)$ since otherwise the generator does not have real coefficients.

\subsubsection{Special case: $c=0\neq a_2$}

We again use Maple to obtain the solutions. One of the solutions is a~very complicated one that can be expressed in terms of roots of some polynomial equation of degree nine. We will not study it further. Then we have a~solution of the form
\begin{equation}
a_1=0,\quad a_2=\delta^2,\quad b=0,\quad d=-2\delta,\quad e=4.
\end{equation}

Finally, there is a~solution where $a_1$ and~$a_2$ are arbitrary and $b=d=e=0$. This solution is somehow obvious~-- the category $\langle a_1\fourpart+a_2\Labab\rangle_\delta$ can never contain $\singleton\otimes\singleton$ since all blocks of both $\Laaaa$ and~$\Labab$ have even size. This, however, says nothing about its non-easiness, so let us use our algorithm to investigate the category.

For simplicity, we can divide the generator by~$a_2$ (for $a_2=0$ is the category obviously easy), that is, consider $p:=\Labab+a\,\fourpart$. After one iteration of \textsc{AddTensors}, we see that $\langle p\rangle_\delta$ may be non-easy only if $a=-2$.\footnote{Details of the computation at \url{https://nbviewer.jupyter.org/github/gromadan/partcat/blob/master/lincat/Section_4_3_4.ipynb}}

\begin{cand}
\label{Cand3}
We have two new candidates for non-easy categories
$$\langle\delta^2\Labab-2\delta(\Labac+\Labcb)+4\Labcd\rangle_\delta,\quad\hbox{and}\quad\langle \Labab-2\Laaaa\rangle_\delta.$$
\end{cand}

The non-easiness of both is proven by Proposition \ref{P.disjoin} and~\ref{P.join}, respectively. Moreover, we will prove that both are actually isomorphic to the category of all pairings $\Pair_\delta=\langle\crosspart\rangle_\delta$.

\subsection{Generator of length four, case with singletons}

In this subsection, we assume $\singleton\otimes\singleton\in\Kat$, so we can assume $p$ is of the form
\begin{equation}
\label{eq.p2}
p=a_1\Laaaa+a_2\Labab+b_1\Laaab+b_2\Labbb+b_3\Labaa+b_4\Laaba+d_1\Labac+d_2\Labcb.
\end{equation}
We do not include $\Labcd$ and rotations of~$\Laabc$ in the linear combination since those are generated by $\singleton\otimes\singleton$.

\begin{prop}
Consider $\delta\in\C\setminus\{0,2\}$. Let $p$ be of the form~\eqref{eq.p2}. Suppose $\Kat:=\langle \singleton\otimes\singleton, p\rangle_\delta$ is non-easy and $\Labac\not\in\Kat$. Then $p$ is rotationally symmetric.
\end{prop}
\begin{proof}\footnote{Details of the computation at \url{https://nbviewer.jupyter.org/github/gromadan/partcat/blob/master/lincat/Prop_4_12.ipynb}}
Assume
$$0\neq(R-1)p=:\tilde p=\tilde b_1\Laaab+\tilde b_2\Labbb+\tilde b_3\Labaa-(\tilde b_1+\tilde b_2+\tilde b_3)\Laaba+d(\Labac-\Labcb).$$
We will prove that $\langle\tilde p,\singleton\otimes\singleton\rangle_\delta=\langle\Laaab,\Labac\rangle_\delta$ (which contains all partitions on four points except for~$\Labab$). This already implies that $\langle p,\singleton\otimes\singleton\rangle_\delta$ either equals to $\langle\Laaab,\Labac\rangle_\delta$ or to $\langle\Laaab,\Labac,\Labab\rangle_\delta$, so it is easy.

After one iteration of \textsc{AddTensor} on $\langle\singleton\otimes\singleton,\tilde p\rangle_\delta$, we see that $\Labac\in\langle\singleton\otimes\singleton,\tilde p\rangle_\delta$ assuming $\delta\neq 2$. Hence, we can set $d=0$ and repeat the algorithm for $\langle\tilde p,\Labac\rangle_\delta$. After one iteration of \textsc{AddTensor}, we generate~$\Laaab$ assuming $\delta\neq 0$.
\end{proof}

\subsubsection{Assuming $\protect\Labac\not\in\Kat$}

Take
\begin{equation}
\label{eq.p22}
p=a_1\Laaaa+a_2\Labab+b(\Laaab+\Labbb+\Labaa+\Laaba)+d(\Labac+\Labcb).
\end{equation}
Running one iteration of \textsc{AddTensor} on $\langle\singleton\otimes\singleton, p\rangle_\delta$, we compute $a_1=-b\delta$, $a_2=-b-d\delta$. Further iterations of \textsc{AddTensor} suggest that this category indeed does not contain~$\Labac$ and is indeed non-easy for any $b,d\in\C$.\footnote{Details of the computation at \url{https://nbviewer.jupyter.org/github/gromadan/partcat/blob/master/lincat/Section_4_4_1.ipynb}}

We can write $p=a_1p_1+a_2p_2$, where (assuming $\delta\neq 0$)
\begin{align}
p_1&=\Laaaa-{1\over\delta}(\Laaab+\Labbb+\Labaa+\Laaba)+{1\over\delta^2}(\Labac+\Labcb),\label{eq.p221}\\
p_2&=\Labab-{1\over\delta}(\Labac+\Labcb).\label{eq.p222}
\end{align}
In Propositions~\ref{P.Pcnoneasy} and~\ref{P.Pcnoneasy2}, we will show that the categories $\langle p_1\rangle_\delta$, $\langle p_2\rangle_\delta$ and $\langle p_1,p_2\rangle_\delta$ are indeed noneasy (note that $p_1$ essentially coincides with $\Pc_{(\delta)}\Laaaa$ and $p_2$~essentially coincides with $\Pc_{(\delta)}\Labab$, where $\Pc_{(\delta)}$ will be defined in Def.~\ref{D.Pc}).

\begin{rem}
\label{R.p2}
It actually holds that $\langle p\rangle_\delta=\langle p_1,p_2\rangle_\delta$ for any non-trivial linear combination $p=a_1p_1+a_2p_2$. For most choices of $a_1$,~$a_2$, this can be computed with our algorithm. However, choosing
\begin{align*}
p&=2\delta p_1+(2-\delta)p_2\\&=2\delta\Laaaa+(2-\delta)\Labab-2(\Laaab+\Labbb+\Labaa+\Laaba)+\Labac+\Labcb,
\end{align*}
the category $\langle p\rangle_\delta$ appears to be new. That is, even if we iterate \textsc{AddTensor} until the modules become stable, we do not obtain $p_1$ and~$p_2$. As we mentioned at the beginning of this remark, in reality, the $p_1$ and $p_2$ are elements of the category. We can compute it ``by hand'' (preferably, again with the help of computer), if we do the computation described by the following graph.\footnote{Details of the computation at \url{https://nbviewer.jupyter.org/github/gromadan/partcat/blob/master/plin/Remark_4_13.ipynb}}
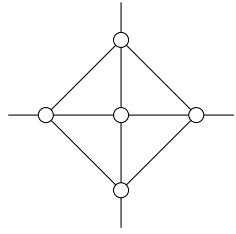
$$
\begin{tikzpicture}
\draw (-1.5,0) -- (1.5,0);
\draw (0,-1.5) -- (0,1.5);
\draw (-1,0) -- (0,-1) -- (1,0) -- (0,1) -- cycle;
\filldraw [fill=white]
   (0,0) circle (0.1)
   (0,1) circle (0.1)
   (0,-1) circle (0.1)
   (-1,0) circle (0.1)
   (1,0) circle (0.1);
\end{tikzpicture}
$$
Here the vertices stand for copies of the generator~$p$~-- each vertex has degree four as $p$ is a~linear combination of partitions of four points~-- and the edges describe contractions (free edges connected just to one vertex are the outputs). One can check that it is indeed possible to perform this computation using the category operations. The key point is that this graph is planar.

So, if we do such a~computation, the result is
\begin{multline*}
(2\delta^5-18\delta^4+48\delta^3-48\delta^2+96\delta-64)\delta p_1\\-(\delta^6-11\delta^5+50\delta^4-144\delta^3+304\delta^2-320\delta+64)p_2+\dots,
\end{multline*}
where the dots stand for some partitions that are already generated by $\singleton\otimes\singleton$. For $\delta\neq 0,2,3,4$, this is a~different linear combination than we started with, so we can indeed generate $p_1$ and~$p_2$.
\end{rem}

\subsubsection{Assuming $\protect\Labac\in\Kat$}

In this case, we are interested in categories of the form $\Kat:=\langle \Labac,p\rangle_\delta$, where
$$p=a_1\Laaaa+a_2\Labab+b(\Laaab+\Labbb+\Labaa+\Laaba).$$
Using our algorithm, it can be again proven that non-easiness implies $a_1=-b\delta$. So, our candidates are quantum groups of the form $\langle a_1p_1+a_2p_2,\Labac\rangle_\delta=\langle a_1p_1+a_2\Labab,\Labac\rangle_\delta$, where $p_1$ and~$p_2$ are given by Eqs.~\eqref{eq.p221},~\eqref{eq.p222} (this time, we can also ignore the summands $\Labac$,~$\Labcb$ in the formulae \eqref{eq.p221},~\eqref{eq.p222}).\footnote{Details of the computation at \url{https://nbviewer.jupyter.org/github/gromadan/partcat/blob/master/lincat/Section_4_4_2.ipynb}}

Again, our algorithm shows, that actually $\langle a_1p_1+a_2\Labab,\ppen\Labac\rangle_\delta=\langle p_1,\Labab,\ppen\Labac\rangle_\delta$ for most choices of $a_1,a_2$. From Remark~\ref{R.p2}, it actually follows that we have this for all $a_1,a_2\neq 0$ if we assume $\delta\neq 0,\ppen 2,3,4$.

Finally, let us mention that we can, in addition, construct the non-easy categories of the form $\langle \singleton,p\rangle$. Again, see Proposition~\ref{P.Pcnoneasy2}.

\begin{cand}
\label{Cand4}
Consider the following candidates for non-easy categories.
\begin{align}
\label{eq.Ksing}&\langle p_1,\singleton\otimes\singleton\rangle_\delta &\qquad&\langle p_1,p_2,\singleton\otimes\singleton\rangle_\delta\\
\label{eq.Kabac}&\langle p_1,\Labac\rangle_\delta &\qquad&\langle p_1,\Labac,\Labab\rangle_\delta\\
&\langle p_1,\singleton\rangle_\delta\hfill &\qquad&\langle p_1,\singleton,\Labab\rangle_\delta
\end{align}
Here, $p_1,p_2$ are given by Eqs.~\eqref{eq.p221},~\eqref{eq.p222}. Assuming $\delta\neq 0,\ppen 2,3,4$, those on lines \eqref{eq.Ksing},~\eqref{eq.Kabac} are the only non-easy categories containing $\singleton\otimes\singleton$ generated by a~single element of $\Partlin_\delta(4)$.
\end{cand}

\section{Concluding remarks on the use of our algorithm}

Let us highlight the contribution of the presented computations to the research in compact quantum groups and suggest some directions for further research.

We were able to find several new examples of partition categories. Most of were interpreted within the theory of compact quantum groups in a separate article \cite{GW18}. Some of the categories that are left over, namely Candidates~\ref{Cand3} are interpreted in Section~\ref{sec.twists} on anticommutative twists. In addition, all the candidates are proven non-easy without any reference to the theory of quantum groups in Section~\ref{sec.noneasydirect}.

As for the size of the considered partitions, the computations presented in Section~\ref{sec.results} are almost at the limit of what can be achieved using our naive algorithm. Due to exponentially increasing requierements for memory and time, we cannot increase the value of the length bound $l_0$ too much. In Remark~\ref{R.p2}, we have seen that even if we choose the length bound to be twice the size of our generator, it may happen that the category approximation is not precise enough. In this case, two categories were incorrectly determined to be distinct although they were not. Nonetheless, we believe that computer algebra might still be useful for seeking new categories of partitions if we make some further assumptions on our categories.

Note for example that all the interesting categories we constructed here are generated by a rotationally-symmetric linear combination of partitions. When looking for other examples of non-easy categories, it may be convenient to focus on rotationally-symmetric generators.

Secondly, we believe that computer algebra might be useful to attack some concrete hypotheses such as the following. (See \cite{BBCC13,Ban18}.)
\begin{enumerate}
\renewcommand{\theenumi}{\alph{enumi}}
\item Is there a quantum group $G$ such that $S_N\subsetneq G\subsetneq S_N^+$? Equivalently, is there a category $\Kat$ such that $\Partlin_N\supsetneq\Kat\supsetneq\NCPart_N$?
\item Is there a quantum group $G$ such that $O_N^*\subsetneq G\subsetneq O_N^+$? Equivalently, is there a category $\Kat$ such that $\langle\halflibpart\rangle_N\supsetneq\Kat\supsetneq\NCPair_N$?
\end{enumerate}

\section{Direct proofs of non-easiness}
\label{sec.noneasydirect}

In this section, we provide proofs of non-easiness of the categories discovered by computer experiments as described in the previous section. The fact that the discovered categories are new and hence non-easy can be proven by interpreting them in terms of quantum group. This was partially done in \cite{GW18}; some additional instances are also interpreted in this article in Section~\ref{sec.twists}. Nevertheless, we also decided to formulate direct proofs here without reference to the quantum group theory. We formulate this section not only to really prove the statements, but also to show different kinds of proof techniques connected with non-easy quantum groups and to show interesting isomorphisms between different linear categories of partitions.

\subsection{General contractions}
\label{secc.gencont}

The following proof works only for one specific category, whose non-easiness is possible to proof also by other means (see Sect.~\ref{secc.P}). Nevertheless, we consider the proof technique to be quite interesting, so we decided to include it here. The basic idea of proof is the following. Suppose $p$ is reflection symmetric. If $p\in\Partlin_\delta(l)$ generates $p'\in\Partlin_\delta(l')$, this means that $p'$ was made from~$p$ by a~series of tensor products, contractions and rotations. We can simplify this process a~bit. First, we produce a~$k$-fold tensor product $p^{\otimes k}$ and then perform some more general contractions. Namely, we can express $p'=qp^{\otimes k}$, where $q\in\Pair_\delta(l'k,l)$ is some pairing. In fact, we can generate any element of $\langle p,\crosspart\rangle_\delta$ by such a~process.

\begin{prop}
\label{P.contractions}
The category
$$\langle \delta^2\,\Laaa-\delta(\Labb+\Laab+\Laba)+2\Labc\rangle_\delta$$
is non-easy for every $\delta\in\C$.
\end{prop}
\begin{proof}
Denote $p:=\delta^2\,\Laaa-\delta(\Labb+\Laab+\Laba)+2\Labc$, $\Kat:=\langle p\rangle_\delta$. We showed in Lemma~\ref{L.3crit} that $\Kat$ is non-easy if and only if $\singleton\not\in\langle p\rangle_\delta$. If we had $\singleton\in\Kat$, it would mean that by a~series of tensor products, contractions and rotations, we can produce a~non-zero multiple of~$\singleton$ from~$p$. Thanks to $p$ being rotationally invariant, this would imply that there exists $k\in\N$ and $q'\in\Pair_\delta(3k-1,0)$ such that $(\idpart\otimes q')p^{\otimes k}=\alpha\singleton$ for some $\alpha\neq 0$. Consequently, $(\idpart\otimes q)p=\alpha\singleton$ for $q:=q'(\idpart\otimes\idpart\otimes p^{\otimes(k-1)})\in\Partlin_\delta(2,0)$.

Hence, it is enough to prove that $qp=0$ for every $q\in\Partlin_\delta(2,0)$. In Section~\ref{secc.len3}, we checked that $(\idpart\otimes\uppairpart)p=\Pi_2 p=0$. It is straightforward to check that also $(\idpart\otimes\upsingleton\otimes\upsingleton)p=0$.
\end{proof}

Actually, this also proves non-easiness of the category $\langle p,\crosspart\rangle_\delta$.

We could formulate a~more general statement such as: Let $p\in\Partlin_\delta(l)$ be rotation and reflection symmetric with $l$ odd. Suppose $(\idpart\otimes q)p=0$ for every $q\in\Partlin_\delta(l-1,0)$. Then $\langle p\rangle_\delta$ and $\langle p,\crosspart\rangle_\delta$ are non-easy categories.

This might sound like a~promising way of constructing new non-easy categories. We only have to solve some system of linear equations. For dimension reasons, we actually surely will have a~plenty of solutions. However, it might happen that all the discovered non-easy categories are actually equal to the above mentioned one. This is at least the case for $l=5$.

\subsection{Isomorphism by conjugation}
\label{secc.tau}

This subsection basically follows \cite[Section~7]{GW18}. We assume $\delta\neq 0$ here.

\begin{defn}
We define the linear combination $\tau_{(\delta)}:=\idpart-{2\over\delta}\disconnecterpart\in\Partlin_\delta(1,1)$. For any $p\in\Partlin_\delta(k,l)$, we set $\Tc_{(\delta)}p:=\tau_{(\delta)}^{\otimes l}p\tau_{(\delta)}^{\otimes k}$.
\end{defn}

It holds that $\tau_{(\delta)}\cdot\tau_{(\delta)}=\idpart$ and $\tau_{(\delta)}^*=\tau_{(\delta)}$. In operator language, we would say that $\tau_{(\delta)}$ is a~self-adjoint unitary. Consequently, conjugation by~$\tau_{(\delta)}$ defines a~category isomorphism.

\begin{prop}
$\Tc_{(\delta)}$ is a~monoidal $*$\hbox{-}isomorphism $\Partlin_\delta\to\Partlin_\delta$.
\end{prop}
\begin{proof}
The fact that $\Tc_{(\delta)}$ is a~monoidal unitary functor follows from the above mentioned properties of~$\tau_{(\delta)}$. Finally, we also have $\Tc_{(\delta)}^2=\id$, which proves that it is an isomorphism.
\end{proof}

\begin{rem}
If $\singleton\otimes\singleton\in\Kat\subset\Partlin_\delta$, then $\Tc_{(\delta)}\Kat=\Kat$. Indeed, $\singleton\otimes\singleton\in\Kat$ implies $\tau_{(\delta)}\in\Kat$ and hence $\Tc_{(\delta)}\Kat\subset\Kat$. From the isomorphism property, we have the equality. This implication cannot be reversed. For example, we have $\Tc_{(\delta)}\crosspart=\crosspart$, so $\Tc_{(\delta)}\langle\crosspart\rangle_\delta=\langle\crosspart\rangle_\delta$ although $\singleton\otimes\singleton\not\in\langle\crosspart\rangle_\delta$.
\end{rem}

As a~non-trivial example, we can compute that
\begin{align*}
\Tc_{(\delta)}\fourpart&=\fourpart-{2\over \delta}(\Laaab+\Laaba+\Labaa+\Labbb)\\&+{4\over \delta^2}(\Laabc+\Labac+\Labbc+\Labca+\Labcb+\Labcc)-{16\over \delta^3}\Labcd.
\end{align*}

\begin{prop}
\label{P.Tcnoneasy}
\cite[Example~7.6]{GW18}
The category $\langle\Tc_{(\delta)}\fourpart\rangle_\delta$ is non-easy. In particular,  $\langle\Tc_{(\delta)}\fourpart\rangle_\delta\neq\langle\fourpart\rangle_\delta$.
\end{prop}
\begin{proof}
The inequality follows simply from the fact that $\Tc_{(\delta)}\fourpart\not\in\langle\fourpart\rangle_\delta$. If the category $\langle\Tc_{(\delta)}\fourpart\rangle_\delta$ was easy, then it would contain~$\fourpart$ and be strictly larger than $\langle\fourpart\rangle_\delta$, which would contradict $\Tc_{(\delta)}$ being a~category isomorphism.
\end{proof}

\subsection{The disjoining isomorphism}
\label{secc.disjoin}

\begin{prop}
\label{P.disjoin}
The category $\Kat:=\langle\crosspart-{2\over\delta}(\Pabac+\Pabcb)+{4\over\delta^2}\Pabcd\rangle_\delta$ is isomorphic to $\langle\crosspart\rangle_\delta$ for every $\delta\neq 0$. Consequently, it is a~non-easy category.
\end{prop}
\begin{proof}
We give an explicit formula for the isomorphism $\Dc\colon\Pair_\delta\to\Kat$ acting on any $p\in\Pair_\delta$ as follows.  Every pair block that has between its legs an odd number of points is replaced by $\langle\hbox{pair}\rangle-{2\over\delta}\langle\hbox{singletons}\rangle$. (We use cyclical order of points in the partition. Since all pair partitions have even number of points, it does not matter from which side we count.) For example,
$$\displaylines{\Labab\mapsto \Labab-{2\over\delta}(\Labac+\Labcb)+{4\over\delta^2}\Labcd,\cr
\Pabab\mapsto \Pabab-{2\over\delta}(\Pabac+\Pabcb)+{4\over\delta^2}\Pabcd,\cr
\LPartition{}{0.3:1,4;1:2,6;0.65:3,5}\mapsto \LPartition{}{0.3:1,4;1:2,6;0.65:3,5}-{2\over\delta}\LPartition{0.3:2,6}{0.6:1,4;1:3,5}-{2\over\delta}\LPartition{0.3:3,5}{0.6:1,4;1:2,6}+{4\over\delta^2}\LPartition{0.4:2,3,5,6}{0.8:1,4}.}$$
Now, we only have to prove that it indeed is a~monoidal $*$\hbox{-}isomorphism.

The proof becomes more clear if we formulate it for partitions with lower points only. In order to check that $\Dc$ is indeed a~monoidal unitary functor, we have to prove that it commutes with the one-line operations. This is clear for the tensor product, rotation, and reflection. Now, let us prove that for any $p\in\Pair_\delta(0,k)$, we have $\Dc(\Pi_1p)=\Pi_1(\Dc p)$. We can assume that $p$ is a~partition, not a linear combination. If $p=\pairpart\otimes q$, then the statement is clear, so assume that the first two points of $p$ belong to different blocks. We call a~pair block {\em even} if it has an even number of points between its legs, otherwise it is {\em odd}.

If the blocks corresponding to the first two points of $p$ are even, then by contracting them, we get an even block. The mapping $\Dc$ acts on even blocks as the identity, so it clearly commutes with the contraction. When contracting an even block with an odd block, we get an odd block. Odd blocks are mapped to $\pairpart-{2\over\delta}\singleton\otimes\singleton=\Lrot\tau_{(\delta)}$ by $\Dc$. When contracting $\Lrot\tau_{(\delta)}$ with a~normal block $\pairpart$, we get $\Lrot\tau_{(\delta)}$, so everything is fine also in this case. Finally if both the blocks are odd, then by contracting them, we get an even block. Also when contracting $\Lrot\tau_{(\delta)}$ with another copy of $\Lrot\tau_{(\delta)}$, we get simply $\pairpart$.

Finally, non-easiness of the category follows directly from the explicit description of its elements: if the category was easy, then it would be equal to $\langle\crosspart,\Pabac,\Pabcb,\Pabcd\rangle_\delta=\langle\crosspart,\singleton\otimes\singleton\rangle_\delta$, which is surely larger and hence non-isomorphic with $\langle\crosspart\rangle_\delta$.
\end{proof}

\begin{rem}
\label{R.noneasyiso}
At first sight, it might appear a~bit confusing that we prove the non-easiness of a~category by showing that it is isomorphic to an easy category. But note that the {\em easiness} and {\em non-easiness} are by no means some fundamental abstract characterizations of the categories. It just says whether we chose a~convenient or an inconvenient way how do describe them. The whole point of Section~\ref{sec.noneasydirect} is to express non-easy categories in terms of the easy ones. That is, to find a~convenient description of categories that were defined inconveniently using linear combinations of partitions.
\end{rem}

\subsection{The joining isomorphism}
\label{secc.join}

\begin{prop}
\label{P.join}
The category $\Kat:=\langle 2\connecterpart-\crosspart\rangle_\delta$ is isomorphic to $\langle\crosspart\rangle_\delta$ for every $\delta\in\C$. Consequently, it is a~non-easy category.
\end{prop}
\begin{proof}
We give an explicit formula for the isomorphism $\Pair_\delta\to\Kat$ acting on a~pair partition~$p$ as follows. Every crossing in~$p$ is replaced by $-\langle\hbox{crossing}\rangle+2\langle\hbox{a single block}\rangle$ (by a~single block we mean, that the two blocks that were crossing are united). To be more precise, let $\la_1,\dots,\la_k$ be the set of blocks of~$p$ and denote by~$X_p$ the set of pairs $\{\la_i,\la_j\}$ that cross each other. Then we define
$$\Jc p:=(-1)^{|X_p|}\sum_{\Xi\subset X_p}(-2)^{|\Xi|}p_{\Xi},$$
where $p_{\Xi}$ is created from~$p$ by unifying the pairs of blocks in~$\Xi$.

For example, we map
$$\displaylines{\Pabab\mapsto-\Pabab+2\Paaaa,\cr
\LPartition{}{0.3:1,4;1:2,6;0.65:3,5}\mapsto \LPartition{}{0.3:1,4;1:2,6;0.65:3,5}-2\LPartition{}{1:1,2,4,6;0.65:3,5}-2\LPartition{}{1:2,6;0.65:1,3,4,5}+4\LPartition{}{0.8:1,2,3,4,5,6}.}$$
The second example in word representation reads
\begin{align*}
\Jc\mathsf{abcacb}&=p-2p_{\{\{\la,\lb\}\}}-2p_{\{\{\la,\lc\}\}}+4p_{\{\{\la,\lb\},\{\la,\lc\}\}}\\&=\mathsf{abcacb}-2\,\mathsf{aacaca}-2\,\mathsf{abaaab}+4\,\mathsf{aaaaaa}.
\end{align*}

Now, we only have to prove that it indeed is a~monoidal $*$\hbox{-}isomorphism. We will do this working with partitions on one line. It is clear that the mapping commutes with the tensor product, rotation and reflection. We need to prove this for contraction.

Take a~pair partition~$p$ on $k+2$ points. If $p=\pairpart\otimes q$, then it is easy to see that indeed $\Pi_1(\Jc p)=\Jc(\Pi_1p)=\Jc q$. Now, suppose that the first two points of~$p$ belong to different blocks. Denote the first block by letter~$\la$ and the second block by letter~$\lb$, so the word representation of~$p$ is $p=\la\lb\,x_1x_2\cdots x_k$. Denote $q:=\Pi_1p$ and its word representation $q=\tilde x_1\tilde x_2\cdots\tilde x_k$, where $\tilde x_i=x_i$ if $x_i\neq\la,\lb$ and $\tilde x_i=\lc$ if $x_i=\la$ or $x_i=\lb$. Then it holds that
\begin{align*}
X_q=&\left\{\{x,y\}\in X_p\bigm| \la,\lb\not\in\{x,y\}\right\}\cup\\
    &\left\{\{\lc,x\}\bigm|\{\la,x\}\in X_p\hbox{ and }\{\lb,x\}\not\in X_p\right\}\cup\\
    &\left\{\{\lc,x\}\bigm|\{\la,x\}\not\in X_p\hbox{ and }\{\lb,x\}\in X_p\right\}.
\end{align*}
Denote by~$\pi$ the embedding $X_q\to X_p$. We will prove that $\Pi_1(\Jc p)=\Jc(\Pi_1p)=\Jc q$.

It is easy to see that
$$\Jc q=(-1)^{|X_q|}\,\Pi_1\left(\sum_{\Xi\subset\pi(X_q)\subset X_p}(-2)^{\left|\Xi\right|}p_\Xi\right).$$
In case when $\{\la,\lb\}\not\in X_p$, we have $|X_q|=|X_p|$, so we can exchange this in the formula above. In case when $\{\la,\lb\}\in X_p$, we have $\Pi_1 p_{\Xi}=\Pi_1 p_{\Xi\cup\{\la,\lb\}}$, so
$$\Jc q=(-1)^{|X_p|}\,\Pi_1\left(\sum_{\Xi\subset\pi(X_q)\subset X_p}(-2)^{\left|\Xi\right|}p_\Xi+\sum_{\Xi\subset\pi(X_q)\subset X_p}(-2)^{\left|\Xi\right|+1}p_{\Xi\cup\{\la,\lb\}}\right).$$

It suffices to prove that the rest of the sum is zero. Choose a~block~$x$ of~$p$ such that $\{\la,x\}\in X_p$ and $\{\lb,x\}\in X_p$. Then
$$\Pi_1\left(p_{\{\{\la,x\}\}}+p_{\{\{\lb,x\}\}}-2p_{\{\{\la,x\},\{\lb,x\}\}}\right)=0.$$
Consequently, for any $\Xi\subset X_p$,
\begin{multline*}
\Pi_1\big((-2)^{\left|\Xi\cup\{\{\la,x\}\}\right|}p_{\Xi\cup\{\{\la,x\}\}}+(-2)^{\left|\Xi\cup\{\{\lb,x\}\}\right|}p_{\Xi\cup\{\{\lb,x\}\}}\\+(-2)^{\left|\Xi\cup\{\{\la,x\},\{\lb,x\}\}\right|}p_{\Xi\cup\{\{\la,x\},\{\lb,x\}\}}\big)=0.
\end{multline*}
The missing part of the sum above is a~sum of such terms, so this proves the statement.

Finally, the non-easiness of the category again follows directly from the explicit description of its elements: if the category was easy, then it would be equal to $\langle\crosspart,\Paaaa\rangle_\delta$, which is larger and hence non-isomorphic with $\langle\crosspart\rangle_\delta$.
\end{proof}

\begin{rem}
\label{R.halflibiso}
We can apply this isomorphism also on subcategories of $\Pair_\delta$. The only easy subcategories are the following two. The category of all non-crossing pairings $\langle\rangle_\delta$, where the isomorphism acts as the identity since there is no crossing, and the category $\langle\halflibpart\rangle_\delta$ that is mapped onto the following non-easy category
$$\langle\halflibpart-2\Pabaaba-2\Paabaab-2\Pabbabb+4\Paaaaaa\rangle_\delta.$$
This leads to an additional new non-easy category that was not discovered by our computer experiments since it is generated by a~partition of six points.
\end{rem}

\subsection{Projective morphism}
\label{secc.P}

Regarding this section, most of the work was done in~\cite{GW18}. Therefore, we present the results somewhat briefly here. For a more self-contained version of the text, we refer to the first author's PhD thesis \cite{thesis}.

We assume $\delta\neq 0$ in the whole section.

\begin{defn}
We define $\pi_{(\delta)}:=\idpart-{1\over\delta}\disconnecterpart\in\Partlin_\delta(1,1)$.
\end{defn}

It satisfies $\pi_{(\delta)}\cdot\pi_{(\delta)}=\pi_{(\delta)}$ and $\pi_{(\delta)}^*=\pi_{(\delta)}$. In operator language, $\pi_{(\delta)}$~is an orthogonal projection. This allows us to define the following

\begin{defn}
\label{D.Pc}
For any $p\in\Partlin_\delta(k,l)$ we denote $\Pc_{(\delta)}p:=\pi_{(\delta)}^{\otimes l}p\pi_{(\delta)}^{\otimes k}$. We denote
$$\PartRed_\delta(k,l):=\Pc_{(\delta)}\Partlin_\delta(k,l)=\{\pi_{(\delta)}^{\otimes l}p\pi_{(\delta)}^{\otimes k}\mid p\in\Partlin_\delta(k,l)\}.$$
\end{defn}

The collection of vector spaces $\PartRed_\delta(k,l)$ is closed under the category operations. It forms a~monoidal $*$\hbox{-}category with identity morphism $\pi_{(\delta)}^{\otimes k}\in\PartRed_\delta(k,k)$ and duality morphisms $\Lrot^k\pi_{(\delta)}^{\otimes k}\in\PartRed(0,2k)$.

\begin{ex}
\label{ex.P}
As an example, let us compute the action of~$\Pc_{(\delta)}$ on small block partitions:
\begin{align*}
\Pc_{(\delta)}\singleton&=0,\\
\Pc_{(\delta)}\pairpart&=\pairpart-{1\over \delta}\singleton\otimes\singleton=\Lrot\pi_{(\delta)},\\
\Pc_{(\delta)}\Laaa&=\Laaa-{1\over \delta}(\Laab+\Laba+\Labb)+{2\over \delta^2}\Labc,\\
\Pc_{(\delta)}\Laaaa&=\Laaaa-{1\over \delta}(\Laaab+\Laaba+\Labaa+\Labbb)+\\&\qquad+{1\over \delta^2}(\Laabc+\Labac+\Labca+\Labbc+\Labcb+\Labcc)-{3\over \delta^3}\Labcd.
\end{align*}
\end{ex}

\begin{defn}
Any collection of spaces $\Kat(k,l)\subset\PartRed(k,l)$ containing $\pi_{(\delta)}$ and $\Lrot\pi_{(\delta)}$ that is closed under the category operations will be called a~\emph{reduced linear category of partitions}. For given $p_1,\dots,p_n\in\Partlin_\delta$, we denote by $\langle p_1,\dots,p_n\rangle\dred$ the smallest reduced category containing those {\em generators}.
\end{defn}

\begin{rem}
\label{R.Pc}\leavevmode
\begin{enumerate}
\renewcommand{\theenumi}{\alph{enumi}}
\item Neither the inclusion $\PartRed_\delta(k,l)\subset\Partlin_\delta(k,l)$ nor the mapping $\Pc_{(\delta)}\colon\Partlin_\delta\to\PartRed_\delta$ are functors.
\item For any $p\in\Partlin_\delta$, we have that $\Pc_{(\delta)}p=p+q$, where $q$ is a~linear combination of partitions containing at least one singleton. In addition, we have $\Pc_{(\delta)}p=0$ whenever $p$ contains a~singleton.
\item For a linear category of partitions $\Kat$ with $\singleton\otimes\singleton\in\Kat$, we have that $\Pc_{(\delta)}\Kat$ is a~reduced category \cite[Proposition~5.11]{GW18}.
\end{enumerate}
\end{rem}

\begin{prop}
\label{P.strictP}
Let $\Kat$ be a~reduced category. Then the following categories are mutually different
$$\langle\Kat\rangle_\delta=\langle\Kat,\singleton\otimes\singleton\rangle_\delta\subsetneq\langle\Kat,\Labac\rangle_\delta\subsetneq\langle\Kat,\singleton\rangle_\delta.$$
\end{prop}
\begin{proof}
The first equality follows from the fact that $\singleton\otimes\singleton$ is a linear combination of $\pi_{\delta}\in\Kat$ and $\pairpart\in\langle\Kat\rangle_\delta$ and hence must be contained in $\langle\Kat\rangle_\delta$. The following two inclusions are obvious, the main point is to prove the strictness. We do this by explicitly describing the categories.

Let $K$ be the set of all $p'\in\Partlin_\delta$ such that $p'$ was made by adding singletons to some $p\in\Kat$. To be more precise, we can formulate this condition recursively: for any $p\in K$, it holds that either $p\in\Kat$ or there is $q\in K$ such that $p$ is some rotation of $q\otimes\singleton$ (including the possibility that $q$ is a~multiple of the empty partition, so $p=\alpha\singleton\in K$ and $p=\alpha\upsingleton\in K$). It is straightforward to prove that $\spanlin K$ is actually a linear category of partitions. For that, see the proof of \cite[Proposition~5.13]{GW18}. From this, it already follows that actually $\spanlin K=\langle\Kat,\singleton\rangle_\delta$.

In the remaining cases, we proceed in a similar way. We can define $K'\subset K$ by choosing only those elements, where we added just an even number of singletons. With similar argumentation, we can show that $\spanlin K'$ is a~category and hence that $\spanlin K'=\langle\Kat,\Labac\rangle_\delta$. Obviously $\singleton\not\in\spanlin K'$, so we have just proven strictness of the second inclusion.

Finally, we define~$K''$ inductively as follows: $K''(k,l)$ contains all elements of $\Kat(k,l)$ and appropriate rotations of $p\otimes\singleton\otimes q\otimes\singleton$ with $p\in\Kat(0,m)$, $q\in\Kat(0,k+l-m-2)$. Again, we can prove that $\spanlin K''$ is a~category equal to $\langle\Kat\rangle_\delta$, which surely does not contain~$\Labac$.
%
\end{proof}

\begin{rem}\cite[Proposition~5.13]{GW18}
\label{R.redP}
The reduced category $\Kat$ can be reconstructed from those by applying $\Pc_{(\delta)}$:
$$\Kat=\Pc_{(\delta)}\langle\Kat\rangle=\Pc_{(\delta)}\langle\Kat,\Labac\rangle_\delta=\Pc_{(\delta)}\langle\Kat,\singleton\rangle_\delta.$$
\end{rem}

\begin{rem}
There are alternative ways of proving the inequalities in Proposition~\ref{P.strictP}. In \cite[Theorem~5.19]{GW18}, the associated quantum groups were characterized and, in particular, shown to be distinct. Hence, the associated categories must also be distinct. Another viewpoint is provided in \cite[Section~7]{GW19}, where additional categories were discovered:
\begin{align*}
  \langle\Kat\rangle_\delta\subsetneq
  \langle\Kat,\LDabcbad\rangle_\delta&\subsetneq
  \langle\Kat,(\LDabac)^{\otimes k}\rangle_\delta\subsetneq
  \langle\Kat,(\LDabac)^{\otimes l}\rangle_\delta\\&\subsetneq
  \langle\Kat,\LDabac\rangle_\delta=
  \langle\Kat,\Labac\rangle_\delta\subsetneq
  \langle\Kat,\singleton\rangle_\delta.
\end{align*}
Here $k>l>1$, $l\mid k$, and
\begin{align*}
\LDabac&=\Labac-\frac{1}{\delta}\Labcd,\\
\LDabcbad&=\Labcbad-\frac{1}{\delta}\Labcbde-\frac{1}{\delta}\Labcdae+\frac{1}{\delta^2}\Labcdef.
\end{align*}
\end{rem}

Now let us mention some concrete categories. Consider the following:
\begin{align*}
\NCPart_\delta&:=\langle\fourpart,\singleton\rangle_\delta=\langle\Laaa\rangle_{\delta}=\spanlin\{\text{all non-crossing partitions}\},\\
\NCPart'_\delta&:=\langle\fourpart,\singleton\otimes\singleton\rangle_\delta=\spanlin\{\text{all non-crossing partitions of even length}\}.
\end{align*}

\begin{lem}[{\cite[Lemma~6.1]{GW18}}]
\label{L.Smin}
It holds that
$$\Pc_{(\delta)}\NCPart_\delta=\langle\Pc_{(\delta)}\Laaa\rangle\dred.$$
\end{lem}

\begin{lem}[{\cite[Lemma~6.2]{GW18}}]
\label{L.SSmin}
Suppose $\delta\neq 3$. It holds that
$$\Pc_{(\delta)}\NCPart'_\delta=\langle\Pc_{(\delta)}\Laaaa\rangle\dred.$$
\end{lem}

\begin{prop}
\label{P.Pcnoneasy}
The categories
$$\begin{matrix}
\langle \Pc_{(\delta)}\Laaa,\singleton\otimes\singleton\rangle_\delta  & \subsetneq & \langle \Pc_{(\delta)}\Laaa,\Labac\rangle_\delta  & \subsetneq & \langle\Pc_{(\delta)}\Laaa,\singleton\rangle_\delta\rlap{${}=\NCPart_\delta$}\\
\rotatebox{90}{$\subsetneq$}                                                 &         & \rotatebox{90}{$\subsetneq$}                     & & \rotatebox{90}{$\subsetneq$}\\
\langle \Pc_{(\delta)}\Laaaa,\singleton\otimes\singleton\rangle_\delta & \subsetneq & \langle \Pc_{(\delta)}\Laaaa,\Labac\rangle_\delta & \subsetneq & \langle \Pc_{(\delta)}\Laaaa,\singleton\rangle_\delta
\end{matrix}\kern5em$$
are mutually distinct and, except for the top right one, are all non-easy.
\end{prop}
\begin{proof}
Thanks to Lemmata \ref{L.Smin} and~\ref{L.SSmin}, we can replace $\Pc_{(\delta)}\Laaa$ and $\Pc_{(\delta)}\Laaaa$ by $\Kat_1:=\Pc_{(\delta)}\langle\Laaa\rangle_\delta$ and $\Kat_2:=\Pc_{(\delta)}\langle\Laaaa\rangle_\delta$, respectively. Surely $\Kat_1\neq\Kat_2$ since $\Kat_2$ contains no partition of odd length. From Remark~\ref{R.redP}, it follows that we have $\Pc_{(\delta)}\Kat=\Kat_1$ for categories $\Kat$ in the first line, whereas $\Pc_{(\delta)}\Kat=\Kat_2$ for categories $\Kat$ from the second line. Consequently, no category from the first line can be equal to any category from the second line. The strictness of the horizontal inclusions follows from Proposition~\ref{P.strictP}. Finally, if the first or the second category from the first line was easy, then it would contain the singleton~$\singleton$ and hence be equal to the last one. Also if the last category of the second row was easy, then it would be equal to $\NCPart_\delta$. If one of the first two categories was easy then surely adding the singleton would preserve the easiness and hence the last one would be easy.
\end{proof}

For a quantum group interpretation of these categories, see \cite[Proposition~6.4]{GW18}.

\begin{lem}
\label{L.redex}
We have the following equalities.
\begin{enumerate}
\item $\langle\Pc_{(\delta)}\crosspart\rangle\dred=\Pc_{(\delta)}\langle\crosspart,\singleton\otimes\singleton\rangle_\delta=\Pc_{(\delta)}\langle\crosspart,\singleton\rangle_\delta$.
\item $\langle\Pc_{(\delta)}\halflibpart\rangle\dred=\Pc_{(\delta)}\langle\halflibpart,\singleton\otimes\singleton\rangle_\delta$.
\item $\langle\Pc_{(\delta)}\crosspart,\Pc_{(\delta)}\Laaa\rangle\dred=\Pc_{(\delta)}\langle\crosspart,\Laaa\rangle_\delta=\Pc_{(\delta)}\Partlin_\delta=\PartRed_\delta$.
\item $\langle\Pc_{(\delta)}\crosspart,\Pc_{(\delta)}\Laaaa\rangle\dred=\Pc_{(\delta)}\langle\crosspart,\Laaaa,\singleton\otimes\singleton\rangle_\delta$.
\end{enumerate}
These four reduced categories are mutually distinct.
\end{lem}
\begin{proof}
In all cases, the inclusion~$\subset$ is obvious. Below, we explain the inclusions~$\supset$. The proof below also explicitly describes elements of the reduced categories, from which it follows that they are indeed mutually distinct.

In case~(1), denote $\Kat:=\langle\crosspart,\singleton\rangle_\delta$. We need to prove that $\Pc_{(\delta)} p\in\langle\Pc_{(\delta)}\crosspart\rangle_\delta$ for every $p\in\Kat$. Since $\Kat$ is easy, it is enough to prove this for partitions~$p$, which linearly generate~$\Kat$. In addition, we have $\Pc_{(\delta)}p=0$ whenever $p$ contains a~singleton. Thus, it is enough to consider partitions~$p$ not containing singletons. Those are exactly all pairings, i.e.\ partitions $p\in\langle\crosspart\rangle$. Without loss of generality, we can assume that $p$ has lower points only, so $p\in\Pair_\delta(k)$. Such a~pairing~$p$ was made from some non-crossing pairing (such as $\pairpart^{\otimes k/2}$) by permuting its points. The reduced category $\langle\Pc_{(\delta)}\crosspart\rangle\dred$ contains $\Pc_{(\delta)}\pairpart^{\otimes k/2}$. By induction, it is enough to prove that if $\Pc_{(\delta)}p'\in\langle\Pc_{(\delta)}\crosspart\rangle\dred$, then $\Pc_{(\delta)}p\in\langle\Pc_{(\delta)}\crosspart\rangle\dred$, where $p,p'\in\Pair_\delta(k)$ such that $p$ was made from~$p'$ by transposing neighbouring points. This transposition can be realized as $p=qp'$, where $q=\transpart$. The proof is finished by observing that $\pi^{\otimes k}q\pi^{\otimes k}=\pi^{\otimes k}q$, so $\Pc_{(\delta)}p=\pi^{\otimes k}qp'=\pi^{\otimes k}q\pi^{\otimes k}p'=(\Pc_{(\delta)}q)(\Pc_{(\delta)}p')$. Let us illustrate this pictorially:
$$(\Pc_{(\delta)}q)(\Pc_{(\delta)}p')=
\BigPartition{
\Pline(0.5,1.5)(8.5,1.5)
\Pline(0.5,1)(8.5,1)
\Pline(0.5,1)(0.5,1.5)
\Pline(8.5,1)(8.5,1.5)
\Ptext(4.5,1.25){$p'$}
\Pline(1,1)(1,0.7)
\Ptext(2,0.8){$\cdots$}
\Pline(3,1)(3,0.7)
\Pline(4,1)(4,0.7)
\Pline(5,1)(5,0.7)
\Pline(6,1)(6,0.7)
\Ptext(7,0.8){$\cdots$}
\Pline(8,1)(8,0.7)
\Ptext(1,0.5){$\pi$}
\Ptext(3,0.5){$\pi$}
\Ptext(4,0.5){$\pi$}
\Ptext(5,0.5){$\pi$}
\Ptext(6,0.5){$\pi$}
\Ptext(8,0.5){$\pi$}
\Pline(1,0.3)(1,-0.3)
\Ptext(2,0){$\cdots$}
\Pline(3,0.3)(3,-0.3)
\Pline(4,0.3)(5,-0.3)
\Pline(5,0.3)(4,-0.3)
\Pline(6,0.3)(6,-0.3)
\Ptext(7,0){$\cdots$}
\Pline(8,0.3)(8,-0.3)
\Ptext(1,-0.5){$\pi$}
\Ptext(3,-0.5){$\pi$}
\Ptext(4,-0.5){$\pi$}
\Ptext(5,-0.5){$\pi$}
\Ptext(6,-0.5){$\pi$}
\Ptext(8,-0.5){$\pi$}
}
=
\BigPartition{
\Pline(0.5,1.5)(8.5,1.5)
\Pline(0.5,1)(8.5,1)
\Pline(0.5,1)(0.5,1.5)
\Pline(8.5,1)(8.5,1.5)
\Ptext(4.5,1.25){$p'$}
\Pline(1,1)(1,0.2)
\Ptext(2,0.5){$\cdots$}
\Pline(3,1)(3,0.2)
\Pline(4,1)(5,0.2)
\Pline(5,1)(4,0.2)
\Pline(6,1)(6,0.2)
\Ptext(7,0.5){$\cdots$}
\Pline(8,1)(8,0.2)
\Ptext(1,0){$\pi$}
\Ptext(3,0){$\pi$}
\Ptext(4,0){$\pi$}
\Ptext(5,0){$\pi$}
\Ptext(6,0){$\pi$}
\Ptext(8,0){$\pi$}
}
=\Pc_{(\delta)}(qp')=\Pc_{(\delta)}p
$$

The proof in all the remaining cases is similar. The key part is to determine the set of all partitions that are elements of the easy category on the right-hand side and do not contain singletons. In case~(2), the category $\langle\halflibpart,\singleton\otimes\singleton\rangle$ contains all partitions with blocks of size one and two such that both legs of all blocks of size two are either on an even position or on an odd position \cite[Prop.~3.5]{Web13}. So, excluding partitions with singletons, we get exactly elements of $\langle\halflibpart\rangle$. Those can be obtained applying permutations that map even points to even points on the non-crossing pairings. For the induction, we then may use the transposition~$\eventranspart$. In both cases (1) and~(2), it is maybe worth mentioning that the corresponding reduced category is actually isomorphic to $\langle\crosspart\rangle_{\delta-1}$ and $\langle\halflibpart\rangle_{\delta-1}$, respectively; the isomorphism is provided by~$\Vc_{(\delta,\pm)}$ defined in the following section.

For case~(3), we need to prove that $\Pc_{(\delta)} p\in\langle\Pc_{(\delta)}\crosspart,\Pc_{(\delta)}\Laaa\rangle_\delta$ for any partition~$p$ not containing a~singleton. This is again a~permutation $p=qp'$, where $q\in\Pair_\delta(k,k)$ and $p'\in\Partlin_\delta(k)$ is non-crossing partition not containing a~singleton. From Lemma~\ref{L.Smin}, we know that $\Pc_{(\delta)}p'\in\langle\Pc_{(\delta)}\Laaa\rangle_\delta\subset\langle\Pc_{(\delta)}\crosspart,\Pc_{(\delta)}\Laaa\rangle_\delta$ and from item~(1), we have that $\Pc_{(\delta)}q\in\langle\Pc_{(\delta)}\crosspart\rangle\dred\subset\langle\Pc_{(\delta)}\crosspart,\Pc_{(\delta)}\Laaa\rangle\dred$. Finally, again $\Pc_{(\delta)}p=\Pc_{(\delta)}q\Pc_{(\delta)}p$.

Case~(4) is basically the same as case~(3) except that we work with partitions of even size and we need to use Lemma~\ref{L.SSmin}.
\end{proof}

\begin{lem}[{\cite[Proposition~6.6]{GW18}}]
We have the following inclusions.
$$\begin{matrix}
\langle\crosspart,\singleton\otimes\singleton\rangle_\delta&=&\langle\crosspart,\Labac\rangle_\delta&\subsetneq&\langle\crosspart,\singleton\rangle_\delta\\
\rotatebox{90}{$\subsetneq$}&&\rotatebox{90}{$=$}&&\rotatebox{90}{$=$}\\
\langle\Pc_{(\delta)}\crosspart,\singleton\otimes\singleton\rangle_\delta&\subsetneq&\langle\Pc_{(\delta)}\crosspart,\Labac\rangle_\delta&\subsetneq&\langle\Pc_{(\delta)}\crosspart,\singleton\rangle_\delta
\end{matrix}$$
\end{lem}
\begin{proof}
The first row is known from classification of the easy categories, see \cite{Web13}. In the second row, we can replace $\Pc_{(\delta)}\crosspart$ by $\langle\Pc_{(\delta)}\crosspart\rangle\dred$ thanks to Lemma~\ref{L.redex}. The inclusions then follow from Proposition~\ref{P.strictP}. The vertical equalities are then easy to see if we write 
\[\Pc_{(\delta)}\crosspart=\crosspart-{1\over \delta}\Pabac-{1\over \delta}\Pabcb+{1\over \delta^2}\Pabcd.\qedhere\]
\end{proof}

\begin{lem}[{\cite[Remark~6.9]{GW18}}]
We have the following inclusions.
$$\begin{matrix}
\langle\crosspart,\singleton\otimes\singleton\rangle_\delta&=&\langle\crosspart,\Labac\rangle_\delta&\subsetneq&\langle\crosspart,\singleton\rangle_\delta\\
\rotatebox{90}{$\subsetneq$}&&\rotatebox{90}{$=$}&&\rotatebox{90}{$=$}\\
\langle\halflibpart,\singleton\otimes\singleton\rangle_\delta&\subsetneq&\langle\halflibpart,\Labac\rangle_\delta&\subsetneq&\langle\halflibpart,\singleton\rangle_\delta\\
\rotatebox{90}{$\subsetneq$}&&\rotatebox{90}{$\subsetneq$}&&\rotatebox{90}{$\subsetneq$}\\
\langle\Pc_{(\delta)}\halflibpart,\singleton\otimes\singleton\rangle_\delta&\subsetneq&\langle\Pc_{(\delta)}\halflibpart,\Labac\rangle_\delta&\subsetneq&\langle\Pc_{(\delta)}\halflibpart,\singleton\rangle_\delta
\end{matrix}$$
\end{lem}
\begin{proof}
The first two rows are again known from classification of easy categories \cite{Web13}. The last row follows again from Lemma~\ref{L.redex} and Proposition~\ref{P.strictP}. Now we explain the strictness of the vertical inclusions between the second and the last row. For the second and third column, we simply apply~$\Pc_{(\delta)}$. We get $\langle\Pc_{(\delta)}\crosspart\rangle\dred$ for the second row, but $\langle\Pc_{(\delta)}\halflibpart\rangle\dred$ for the third row, which proves the inequality.

As for the first column, we can see that $(\disconnecterpart\otimes\pi_{(\delta)}\otimes\pi_{(\delta)})\halflibpart\in\langle\halflibpart,\singleton\otimes\singleton\rangle_\delta$. We prove that this element cannot be contained in $\langle\Pc_{(\delta)}\halflibpart,\singleton\otimes\singleton\rangle_\delta$. If it was there, then it would be contained also in $\langle\Pc_{(\delta)}\halflibpart,\singleton\rangle_\delta$. Composing with $\upsingleton\otimes\idpart\otimes\idpart$ from left and with $\idpart\otimes\idpart\otimes\singleton$ from right, we get $\Pc_{(\delta)}\crosspart\in\langle\Pc_{(\delta)}\halflibpart,\singleton\rangle_\delta$ and hence $\Pc_{(\delta)}\crosspart\in\langle\Pc_{(\delta)}\halflibpart\rangle\dred\subsetneq\langle\Pc_{(\delta)}\crosspart\rangle\dred$, which is a~contradiction.
\end{proof}

Let us summarize all the non-easy categories we obtained above.

\begin{prop}
\label{P.Pcnoneasy2}
The categories
\vskip-30pt
$$\begin{matrix}
&&&&\Partlin_\delta\\
&&&&\rotatebox{90}{$=$}\\
\langle\Pc_{(\delta)}\crosspart, \Pc_{(\delta)}\Laaa,\singleton\otimes\singleton\rangle_\delta  & \subsetneq & \langle\Pc_{(\delta)}\crosspart, \Pc_{(\delta)}\Laaa,\Labac\rangle_\delta  & \subsetneq & \langle\Pc_{(\delta)}\crosspart,\Pc_{(\delta)}\Laaa,\singleton\rangle_\delta\\
\rotatebox{90}{$\subsetneq$}                                                            &         & \rotatebox{90}{$\subsetneq$}                                     & & \rotatebox{90}{$\subsetneq$}\\
\langle\Pc_{(\delta)}\crosspart, \Pc_{(\delta)}\Laaaa,\singleton\otimes\singleton\rangle_\delta & \subsetneq & \langle\Pc_{(\delta)}\crosspart, \Pc_{(\delta)}\Laaaa,\Labac\rangle_\delta & \subsetneq & \langle\Pc_{(\delta)}\crosspart, \Pc_{(\delta)}\Laaaa,\singleton\rangle_\delta\\
\rotatebox{90}{$\subsetneq$}                                                            &         & \rotatebox{90}{$\subsetneq$}                                     & & \rotatebox{90}{$\subsetneq$}\\
\langle\Pc_{(\delta)}\crosspart,\singleton\otimes\singleton\rangle_\delta & \subsetneq & \langle\Pc_{(\delta)}\crosspart,\Labac\rangle_\delta & \subsetneq & \langle\Pc_{(\delta)}\crosspart,\singleton\rangle_\delta\\
\rotatebox{90}{$\subsetneq$}&&\rotatebox{90}{$\subsetneq$}&&\rotatebox{90}{$\subsetneq$}\\
\langle\Pc_{(\delta)}\halflibpart,\singleton\otimes\singleton\rangle_\delta&\subsetneq&\langle\Pc_{(\delta)}\halflibpart,\Labac\rangle_\delta&\subsetneq&\langle\Pc_{(\delta)}\halflibpart,\singleton\rangle_\delta
\end{matrix}$$
are mutually distinct and, except for the top right one, are all non-easy. They are also distinct from the categories of Proposition~\ref{P.Pcnoneasy}.
\end{prop}
\begin{proof}
Follows directly from all the results above.
\end{proof}

\subsection{Category coisometry}
\label{secc.V}

In this section, we assume $\delta>0$, $\delta\neq1$.

\begin{defn}[{\cite[Def.~4.10, Rem.~4.16]{GW18}}]
\label{D.upsilon}
We define
$$\upsilon_{(\delta-1,\pm)}:=\idpart-{1\over\delta-1}\left(1\pm{1\over\sqrt{\delta}}\right)\in\Partlin_{\delta-1}(1,1).$$
For every $p\in\Part(k,l)$, we define $\Bc_{(\delta)}p\in\Partlin_{\delta-1}(k,l)$ the linear combination that was made from~$p$ by replacing every block of $k$ points by $\langle\text{block}\rangle+(-1)^k\langle\text{singletons}\rangle$. We extend this definition linearly to define a~map $\Bc_{(\delta)}\colon\Partlin_\delta\to\Partlin_{\delta-1}$. Finally, for every $p\in\Partlin_\delta(k,l)$, we define $\Vc_{(\delta,\pm)} p:=\upsilon_{(\delta-1,\pm)}^{\otimes l}(\Bc_{(\delta)}p)\upsilon_{(\delta-1,\pm)}^{\otimes k}$.
\end{defn}

\begin{ex}
As an example, let us mention how $\Vc_{(\delta,\pm)}$ acts on the smallest block partitions.
\begingroup
\begin{align*}
\Vc_{(\delta,\pm)}\singleton&=0,\\
\Vc_{(\delta,\pm)}\pairpart&=\pairpart,\displaybreak[0]\\
\Vc_{(\delta,\pm)}\Laaa&=\Laaa-{1\over \delta-1}\left(1\pm{1\over\sqrt{\delta}}\right)(\Laab+\Laba+\Labb)\\&\hskip16em+{1\over (\delta-1)^2}\left(2\pm{\delta+1\over\sqrt{\delta}}\right)\Labc,\displaybreak[0]\\
\Vc_{(\delta,\pm)}\Laaaa&=\Laaaa-{1\over \delta-1}\left(1\pm{1\over\sqrt{\delta}}\right)(\Laaab+\Laaba+\Labaa+\Labbb)+\\&\qquad+{1\over (\delta-1)^2}\left({\delta+1\over\delta}\pm{2\over\sqrt{\delta}}\right)(\Laabc+\Labac+\Labca+\\&\qquad+\Labbc+\Labcb+\Labcc)+{1\over(\delta-1)^3}\left({\delta^2-6\delta-5\over\delta}\mp{8\over\sqrt{\delta}}\right)\Labcd.
\end{align*}
\endgroup
\end{ex}

\begin{rem}
\label{R.Vc}\leavevmode
\begin{enumerate}
\renewcommand{\theenumi}{\alph{enumi}}
\item We have $\Bc_{(\delta)}\singleton=0$. Consequently, $\Bc_{(\delta)}p=0=\Vc_{(\delta,\pm)}p$ for every $p\in\Part$ containing a~singleton. 
\item As a~consequence of the preceding point and Remark~\ref{R.Pc}(b), we have $\Vc_{(\delta,\pm)}=\Vc_{(\delta,\pm)}\circ\Pc_{(\delta)}$.
\item The mapping $\Vc_{(\delta,\pm)}\colon\Partlin_\delta\to\Partlin_{\delta-1}$ is not a~functor.
\item For any $p\in\Partlin_\delta$, we have that $\Vc_{(\delta,\pm)}p=p+q$, where $q$ is a~linear combination of partitions containing at least one singleton.
\item Consequently, $\Vc_{(\delta,\pm)}$ acts injectively on partitions with no singletons. For the same reason, it also acts injectively on $\PartRed_\delta=\Pc_{(\delta)}\Partlin_\delta$.
\item $\Vc_{(\delta,\pm)}$ acts {\em blockwise} on partitions $p\in\Part$. That is, we may map all the blocks constituting a~given partition $p\in\Part$ an then ``assemble'' the image of~$p$ from the images of the blocks. More formally, we could say that $\Vc_{(\delta)}$ commutes with tensor products and arbitrary permutations of points. It follows from the fact that $\Bc_{(\delta)}$ acts blockwise by definition and conjugating by a~given partition is also a~blockwise operation.
\item We have $\upsilon_{(\delta-1,\pm)}\singleton=\mp{1\over\sqrt{\delta}}\singleton$. So, if $b_k\in\Part(k)$ is a~partition consisting of a~single block, we have
$$\Vc_{(\delta,\pm)}b_k=\upsilon^{\otimes k}(b_k+(-1)^k\singleton^{\otimes k})=\upsilon^{\otimes k}b_k+\left({\pm1\over\sqrt\delta}\right)^k\singleton^{\otimes k}.$$
\end{enumerate}
\end{rem}

\begin{prop}
\label{P.Vfunct}
The mapping $\Vc_{(\delta,\pm)}$ acts on $\PartRed_\delta$ as a~faithful monoidal unitary functor.
\end{prop}
\begin{proof}
The injectivity of $\Vc_{(\delta,\pm)}$ was mentioned in Remark~\ref{R.Vc}(e). It remains to prove the functorial property. We will work with partitions with lower points only. So, we need to prove that $\Vc_{(\delta,\pm)}$ commutes with tensor product, contractions, rotations and reflections. The only non-trivial part are the contractions, the other operations follow from the fact that $\Vc_{(\delta,\pm)}$ acts blockwise as mentioned in Remark~\ref{R.Vc}(f).

Since $\Vc_{(\delta,\pm)}$ acts blockwise, it is enough to prove it for blocks. Denote by $b_k\in\Part(k)$ the partition consisting of a~single block. Then we have to prove that
$$\Pi_1\Vc_{(\delta,\pm)}\Pc_{(\delta)}b_k=\Vc_{(\delta,\pm)}\Pi_1\Pc_{(\delta)}b_k\quad\hbox{and}\quad\Pi_{k-1}\Vc_{(\delta,\pm)}\Pc_{(\delta)}(b_k\otimes b_l)=\Vc_{(\delta,\pm)}\Pi_{k-1}\Pc_{(\delta)}(b_k\otimes b_l).$$
To do that, first note that $\uppairpart(\upsilon_{(\delta-1,\pm)}\otimes\upsilon_{(\delta-1,\pm)})=\uppairpart-{1\over\delta}\upsingleton\otimes\upsingleton$. So, assuming $k>2$, we have\footnote{Pay attention to the fact that the computations take place mostly in $\Partlin_{\delta-1}$, not~$\Partlin_\delta$!}
\begin{align*}
\Pi_1\Vc_{(\delta,\pm)}b_k&=\left(\left(\uppairpart-{1\over\delta}\upsingleton\otimes\upsingleton\right)\otimes\upsilon^{\otimes(k-2)}\right)(b_k+(-1)^k\singleton^{\otimes k})\\
&=\left(1-{1\over\delta}\right)\upsilon^{\otimes(k-2)}(b_{k-2}+(-1)^k\singleton^{\otimes(k-2)})=\Vc_{(\delta,\pm)}\Pi_1\Pc_{(\delta)}b_k.
\end{align*}
For $k=2$, this equality holds as well since $\Pi_1\Vc_{(\delta,\pm)}\pairpart=\Pi_1\pairpart=\delta=\Vc_{(\delta,\pm)}\Pi_1\Pc_{(\delta)}\pairpart$. Now, assuming $k,l>1$, we have
\begin{align*}
&\Pi_{k-1}\Vc_{(\delta,\pm)}(b_k\otimes b_l)\\
&=\left(\upsilon^{\otimes(k-1)}\otimes\left(\uppairpart-{1\over\delta}\upsingleton\otimes\upsingleton\right)\otimes\upsilon^{\otimes(l-1)}\right)\\&\kern0.2\hsize(b_k\otimes b_l+(-1)^k\singleton^{\otimes k}\otimes b_l+(-1)^lb_k\otimes\singleton^{\otimes l}+(-1)^{k+l}\singleton^{\otimes(k+l)})\\
&=\upsilon^{\otimes(k+l-2)}\left(b_{k+l-2}-{1\over\delta}b_{k-1}\otimes b_{l-1}+{(-1)^k\over\delta}\singleton^{\otimes (k-1)}\otimes b_{l-1}+\right.\\&\kern0.2\hsize\left.+{(-1)^l\over\delta}b_{k-1}\otimes\singleton^{\otimes (l-1)}+\left(1-{1\over\delta}\right)(-1)^{k+l}\singleton^{\otimes (k+l-2)}\right)\\
&=\Vc_{(\delta,\pm)}\left(b_{k+l-2}-{1\over\delta}(b_{k-1}\otimes b_{l-1})\right)=\Vc_{(\delta,\pm)}\Pi_{k-1}\Pc_{(\delta)}(b_k\otimes b_l).
\end{align*}
For $k=1$ or $l=1$ both sides are obviously equal to zero.
\end{proof}

\begin{rem}
\label{R.Viso}
Consequently, $\Vc_{(\delta,\pm)}$~defines an isomorphism between any reduced category of partitions $\Kat\subset\PartRed_\delta$ and its image $\Vc_{(\delta,\pm)}\Kat\subset\Partlin_{\delta-1}$. So, also for any ordinary linear category of partitions $\Kat\subset\Partlin_\delta$, we have an isomorphism between $\Pc_{(\delta)}\Kat$ and $\Vc_{(\delta,\pm)}\Kat$.

This statement was already proven by another means (using the functor $p\mapsto T_p$) in \cite[Prop.~5.16]{GW18}.
\end{rem}

\begin{prop}
\label{P.Vcnoneasy}
The categories
$$\langle\Vc_{(\delta,\pm)}\Laaa\rangle_{\delta-1},\qquad\langle\Vc_{(\delta,\pm)}\Laaaa\rangle_{\delta-1}$$
are both non-easy.
\end{prop}
\begin{proof}
From Proposition~\ref{P.Vfunct}, it follows that the above mentioned categories are isomorphic to~$\langle\Pc_{(\delta)}\Laaa\rangle\dred$ and $\langle\Pc_{(\delta)}\Laaaa\rangle\dred$, respectively. If they were easy, they would contain all the summands of $\Vc_{(\delta,\pm)}\Laaa$, resp.\ $\Vc_{(\delta,\pm)}\Laaaa$ (Prop.~\ref{P.summands}) and hence be equal to $\NCPart_{\delta-1}$, resp.\ $\NCPart'_{\delta-1}$,. This cannot happen since
\begin{align*}
\dim\langle\Vc_{(\delta,\pm)}\Laaa\rangle_{\delta-1}(0,3)&=\dim\langle\Pc_{(\delta)}\Laaa\rangle\dred(0,3)\\&<\dim\NCPart_\delta(0,3)=\dim\NCPart_{\delta-1}(0,3),\\
\dim\langle\Vc_{(\delta,\pm)}\Laaaa\rangle_{\delta-1}(0,4)&=\dim\langle\Pc_{(\delta)}\Laaaa\rangle\dred(0,4)\\&<\dim\NCPart'_\delta(0,4)=\dim\NCPart'_{\delta-1}(0,4).\qedhere
\end{align*}
\end{proof}

\section{Anticommutative twists}
\label{sec.twists}

In this section, we interpret the category isomorphisms $\Dc$ and~$\Jc$ described in Sections \ref{secc.disjoin}, \ref{secc.join}. As a~consequence, we are going to interpret the non-easy categories discovered as Candidates~\ref{Cand3}
$$\left\langle\crosspart-{2\over N}(\Pabac+\Pabcb)+{4\over N^2}\Pabcd\right\rangle_N\quad\hbox{and}\quad\langle 2\connecterpart-\crosspart\rangle_N.$$

As we already showed, these categories are both isomorphic to the category of all pairings $\Pair_N$. The idea is that instead of studying the image of these categories under the standard functor $p\mapsto T_p$, we study some alternative functors $p\mapsto\tilde T_p=T_{\Dc p}$, resp.\ $p\mapsto\tilde T_p=T_{\Jc p}$ acting on pair partitions $p\in\Pair_N$. Changing this functor also changes the interpretation of the partitions in terms of relations. In particular, the crossing partition~$\crosspart$, which generates the whole category $\Pair_N$, will no longer imply commutativity. We get some deformed commutativity instead. More concretely, some minus signs will appear, so the commutativity will partially change to anticommutativity.

To keep this section short, we will not recall the basics of the theory of compact quantum groups and Hopf algebras. We refer the reader to the monographs \cite{Tim08,NT13}.

\subsection{2-cocycle deformations}

We recall a construction from \cite{Doi93,Sch96,BY14}.

Let $A$ be a Hopf $*$-algebra. We use the Sweedler notation $\Delta(x)=x_{(1)}\otimes x_{(2)}$. A \emph{unitary 2-cocycle} on $A$ is a convolution invertible linear map $\sigma\colon A\otimes A\to\C$ satisfying
$$\sigma(x_{(1)},y_{(1)})\sigma(x_{(2)}y_{(2)},z)=\sigma(y_{(1)},z_{(1)})\sigma(x,y_{(2)}z_{(2)}),$$
$$\sigma^{-1}(x,y)=\overline{\sigma(S(x)^*,S(y)^*)},$$
and $\sigma(x,1)=\sigma(1,x)=\epsilon(x)$ for $x,y,z\in A$, where $\sigma^{-1}$ denotes the convolution inverse of $\sigma$.

Let $G$ be a compact quantum group and $\sigma$ a 2-cocycle on the associated Hopf $*$-algebra $\Pol G$. Then we can define its deformation $G^\sigma$, where $\Pol G^\sigma$ coincides with $\Pol G$ as a coalgebra and the $*$-algebra structure is defined as follows
\begin{align}
\hat x\hat y&=\sigma(x_{(1)},y_{(1)})\sigma^{-1}(x_{(3)},y_{(3)})\widehat{x_{(2)}y_{(2)}}, \label{eq.defmult}\\
\hat x^*&=\sigma(S(x_{(5)})^*,x_{(4)}^*)\sigma^{-1}(x_{(2)}^*,S(x_{(1)})^*)\widehat{x_{(3)}^*} \label{eq.definv},
\end{align}
where $\hat x$ denotes $x\in\Pol G$ viewed as an element of $\Pol G^\sigma$.

It holds that the quantum groups $G$ and $G^\sigma$ have monoidally equivalent representation categories.

Consider compact quantum groups $H\subset G$, so there is a surjection $q\colon\Pol G\to\Pol H$. A 2-cocycle $\sigma$ on $H$ then induces a 2-cocycle $\sigma_q:=\sigma\circ(q\otimes q)$ on $G$. We often construct 2-cocycles on quantum groups induced by bicharacters on dual discrete quantum subgroups $\hat\Gamma\subset G$.

Let $\Gamma$ be a group. A \emph{unitary bicharacter} on $\Gamma$ is a map $\phi\colon\Gamma\times\Gamma\to\T$ (here $\T$ denotes the complex unit circle) satisfying
$$\phi(xy,z)=\phi(x,z)\phi(y,z),\qquad \phi(x,yz)=\phi(x,y)\phi(x,z).$$
In particular, we have $\phi(x,e)=\phi(e,x)=1$. It is easy to see that any unitary bicharacter $\phi$ on a discrete group $\Gamma$ extends to a unitary 2-cocycle on $\C\Gamma=\Pol\hat\Gamma$. 

\subsection{Anticommutative twists}

We now make a~special choice for~$\sigma$. Consider any $\sigma\in M_N(\{\pm 1\})$. One can easily check that the map $(t_i,t_j)\mapsto\sigma_{ij}$, where $t_1,\dots,t_N$ are generators of~$\Z_2^N$, uniquely extends to a~bicharacter on~$\Z_2^N$. This induces a~2\hbox{-}cocycle on any quantum group~$G$ containing~$\hat\Z_2^N$ as a~quantum subgroup.

So, suppose $G$ is a~compact matrix quantum group with fundamental representation $u\in M_N(\Pol G)$ and $q\colon \Pol G\to \C\Z_2^N$ maps $u_{ij}\mapsto t_i\delta_{ij}$. Let us, for simplicity, restrict to the case $G\subset O_N^+$.

For a~multi-index $\mathbf{i}=(i_1,\dots,i_k)$, denote $\sigma_\mathbf{i}:=\prod_{1\le m<n\le k}\sigma_{i_mi_n}$.

\begin{lem}
\label{L.sigprod}
Suppose, $\bar u=u$, i.e.\ $u_{ij}^*=u_{ij}$. Then
\begin{align}
\hat u_{ij}^*&=\sigma_{ii}\sigma_{jj}\hat u_{ij},\label{eq.oginv}\\
\hat u_{i_1j_1}\cdots\hat u_{i_kj_k}&=\sigma_\mathbf{i}\sigma_\mathbf{j}\,\widehat{(u_{i_1j_1}\cdots u_{i_kj_k})}.\label{eq.ogmult}
\end{align}
\end{lem}
\begin{proof}
Both formulae are obtained simply by using the defining formulae \eqref{eq.defmult}, \eqref{eq.definv}. For the second one, we need to apply induction on~$k$.
\end{proof}

\begin{prop}
Suppose $G\subset O_N^+$. Then $G^\sigma\subset O^+(F)\subset U^+(F)=U^+_N$ with $F_{ij}=\delta_{ij}\sigma_{ii}$.
\end{prop}
\begin{proof}
All the relations are checked using Lemma~\ref{L.sigprod}. The relation $\bar{\hat u}=F^{-1}\hat uF$ is just a~matrix version of Eq.~\eqref{eq.oginv}. Checking the unitarity of $\hat u$ is also straightforward. As an example, let us check the relation $\hat u\hat u^*=1_N$:
$$\sum_k \hat u_{ik}\hat u_{jk}^*=\sum_k\sigma_{jj}\sigma_{kk}\hat u_{ik}\hat u_{jk}=\sum_k \sigma_{ij}\sigma_{jj}\widehat{u_{ik}u_{jk}}=\delta_{ij}.$$
Finally, the fact that $U^+(F)=U_N^+$ follows from $F^*F=1_N$. Indeed, $F\bar{\hat u}F^{-1}$ being unitary can be written as $F\bar{\hat u}F^{-1}(F^*)^{-1}\hat u^tF^*=1_N$ and $(F^*)^{-1}\hat u^tF^*F\bar{\hat u}F^{-1}=1_N$. These relations are obviously equivalent to $\bar{\hat u}\hat u^t=1_N=\hat u^t\bar{\hat u}$.
\end{proof}

Now we analyse the intertwiner spaces for the twisted quantum group $G^\sigma$. This will also prove the equivalence of the representation categories for our special choice of the 2-cocycle.

\begin{prop}
\label{P.sigmor}
Consider $G=(C(G),u)\subset O_N^+$. Then
$$\Mor(\hat u^{\otimes k},\hat u^{\otimes l})=\{T^\sigma\mid T\in\Mor(u^{\otimes k},u^{\otimes l})\}$$
with $T^\sigma_\mathbf{ij}=T_\mathbf{ij}\sigma_\mathbf{i}\sigma_\mathbf{j}$.
\end{prop}
\begin{proof}
If $T\in\Mor(u^{\otimes k},u^{\otimes l})$, it means that $Tu^{\otimes k}=u^{\otimes l}T$, which is certainly equivalent to $T\widehat{u^{\otimes k}}=\widehat{u^{\otimes l}}T$. We can rewrite this in matrix entries as
$$\sum_\mathbf{m}T_\mathbf{im}\,\widehat{(u_{m_1j_1}\cdots u_{m_kj_k})}=\sum_\mathbf{n}\widehat{(u_{i_1n_1}\cdots u_{i_ln_l})}\,T_\mathbf{nj}.$$
Now, applying Lemma~\ref{L.sigprod}, we can rewrite this as
$$\sum_\mathbf{m}{T_\mathbf{im}\over \sigma_\mathbf{m}\sigma_\mathbf{j}}\hat u_{m_1j_1}\cdots \hat u_{m_kj_k}=\sum_\mathbf{n}{T_\mathbf{nj}\over\sigma_\mathbf{i}\sigma_\mathbf{n}}\hat u_{i_1n_1}\cdots \hat u_{i_ln_l}.$$
Finally, using the fact that $\sigma_\mathbf{i},\sigma_\mathbf{j}=\pm 1$, we can see that this is equivalent to $T^{\sigma}\hat u^{\otimes k}=\hat u^{\otimes l}T^\sigma$.
\end{proof}

In connection with partition categories, we can interpret this result as follows. Consider $G:=H_N$ the hyperoctahedral group, which corresponds to the category $\EvenPart_N:=\langle\crosspart,\Paaaa\rangle_N$ spanned by partitions with blocks of even length. It is the smallest partition quantum group having $\hat\Z_2^N$ as a~quantum subgroup. The matrix~$\sigma$ then defines an alternative functor $T^\sigma\colon\EvenPart_N\to\Mat$ mapping $p\mapsto T_p^\sigma$ with $[T_p^\sigma]_\mathbf{ij}=[T_p]_\mathbf{ij}\sigma_\mathbf{i}\sigma_\mathbf{j}=\delta_p(\mathbf{j},\mathbf{i})\sigma_\mathbf{i}\sigma_\mathbf{j}$.

\begin{lem}
The map $T^\sigma\colon\EvenPart_N\to\Mat$ is indeed a~monoidal unitary functor.
\end{lem}
\begin{proof}
Checking that $T^\sigma$ behaves well with respect to composition and involution is straightforward using the fact that $p\mapsto T_p$ is a~monoidal unitary functor. Let us do it for the composition.
\begin{align*}
[T_q^\sigma T_p^\sigma]_\mathbf{ac}&=\sum_\mathbf{b}[T_q^\sigma]_\mathbf{ab}[T_p^\sigma]_\mathbf{bc}=\sum_\mathbf{b}[T_q^\sigma]_\mathbf{ab}[T_p^\sigma]_\mathbf{bc}\sigma_\mathbf{a}\sigma_\mathbf{b}\sigma_\mathbf{b}\sigma_\mathbf{c}\\
&=\sigma_\mathbf{a}\sigma_\mathbf{c}\sum_\mathbf{b}[T_q^\sigma]_\mathbf{ab}[T_p^\sigma]_\mathbf{bc}=[T_{qp}]_\mathbf{ac}\sigma_\mathbf{a}\sigma_\mathbf{c}=[T_{qp}^\sigma]_\mathbf{ac}
\end{align*}
The tensor product is a~bit more complicated. We need to check that
$$\sigma_\mathbf{ac}\sigma_\mathbf{bd}\delta_{p\otimes q}(\mathbf{ac},\mathbf{bd})=\sigma_\mathbf{a}\sigma_\mathbf{b}\sigma_\mathbf{c}\sigma_\mathbf{d}\delta_p(\mathbf{a},\mathbf{b})\delta_q(\mathbf{c},\mathbf{d})$$
for any two partitions $p\in\Part(k,l)$, $q\in\Part(m,n)$ with blocks of even length. We know that $p\mapsto T_p$ is a~monoidal functor, so $\delta_{p\otimes q}(\mathbf{ac},\mathbf{bd})=\delta_p(\mathbf{a},\mathbf{b})\delta_q(\mathbf{c},\mathbf{d})$. Take any $\mathbf{a}$,~$\mathbf{b}$, $\mathbf{c}$,~$\mathbf{d}$ such that $\delta_{p\otimes q}(\mathbf{ac},\mathbf{bd})=1$. We need to show that $\sigma_\mathbf{ac}\sigma_\mathbf{bd}=\sigma_\mathbf{a}\sigma_\mathbf{b}\sigma_\mathbf{c}\sigma_\mathbf{d}$. Equivalently, we need to show that
$$\prod_{i=1}^k\prod_{j=1}^l\prod_{s=1}^m\prod_{t=1}^n\sigma_{a_ic_s}\sigma_{b_jd_t}=1.$$
Recall that we assume that all blocks of $p$ and~$q$ have even size. Consequently, one can check that, for every block~$V$ of~$p$ and every block~$W$ of~$q$, there is an even amount of terms $\sigma_{a_ic_s}$ or $\sigma_{b_jd_t}$ of the product with $i\in V$ and $s\in W$ resp. $j\in V$ and $t\in W$. Since we assume $\delta_{p\otimes q}(\mathbf{ac},\mathbf{bd})=1$, the multiindices $\mathbf{ab}$ and~$\mathbf{cd}$ are constant on the blocks. As a~consequence, the product of those terms always equals one.
\end{proof}

\begin{cor}
\label{C.twistqg}
Let $G$ be a~quantum group group with $H_N\subset G\subset O_N^+$ corresponding to some linear category of partitions~$\Kat$. Then the representation category of~$G^\sigma$ is described by the same partition category~$\Kat$ if one uses the functor~$T^\sigma$ instead of~$T$. That is,
$$\Mor(\hat u^{\otimes k},\hat u^{\otimes l})=\{T^\sigma_p\mid p\in\Kat(k,l)\}.$$
\end{cor}
\begin{proof}
Follows directly from Proposition~\ref{P.sigmor} and the definition of~$T^\sigma_p$.
\end{proof}

\begin{prop}
For any $p\in\EvenPart_N\cap\NCPart_N$, we have $T_p^\sigma=T_p$. In particular, twisting by~$\sigma$ leads to a~new quantum group only for categories with crossings.
\end{prop}
\begin{proof}
It is enough to prove the statement for partitions. Then by linearity of $T$ and~$T^\sigma$, it must hold also for linear combinations.

So, let $p$ be a~non-crossing partition with blocks of even size. It is known that non-crossing partitions are always of the form of some nested blocks. That is, up to rotation, we have $p=q\otimes b$, where $b$ is a~partition consisting of a~single block. Since both $T$ and~$T^\sigma$ are monoidal functors, it is enough to check the statement for block partitions. So, let $b_{2l}\in\Part(0,2l)$ be a~partition with a~single block of $2l$ points. Then indeed
\[[T_{b_2l}^\sigma]_\mathbf{i}=\delta_\mathbf{i}\sigma_\mathbf{i}=\delta_\mathbf{i}\sigma_{i_1}^{2l}=\delta_\mathbf{i}=[T_p]_\mathbf{i}.\qedhere\]
\end{proof}

Crossing partitions correspond to some commutativity relations. The cocycle twist corresponding to the matrix~$\sigma$ then puts some extra signs to the relations, which may make them anticommutative. In particular, it may be interesting to study the relation corresponding to the simple crossing~$\crosspart$, which then reads
\begin{equation}
\label{eq.sigcom}
\sigma_{ik}\sigma_{jl}\,\hat u_{ij}\hat u_{kl}=\sigma_{ki}\sigma_{lj}\,\hat u_{kl}\hat u_{ij}.
\end{equation}

\subsection{Examples}

\begin{ex}
\label{ex.qdef}
If we choose
$$\sigma_{ij}=\begin{cases}-1&i<j\\+1&i\ge j,\end{cases}$$
we get $q$-commutativity for $q=-1$.

Indeed, substituting into Eq. \eqref{eq.sigcom}, we get exactly the defining relation for $O_N^{-1}$
$$u_{ij}u_{ik}=-u_{ik}u_{ij},\quad u_{ji}u_{ki}=-u_{ki}u_{ji}\qquad\text{for $i\neq j$},$$
$$u_{ij}u_{kl}=u_{kl}u_{ij}\qquad\text{for $i\neq k,j\neq l$}.$$

The fact that $O_N^{-1}$ is a cocycle twist of $O_N$ and hence possesses an equivalent representation category was discovered already in \cite[Theorem 4.3]{BBC07}.
\end{ex}

\begin{ex}
\label{ex.grad}
If we choose
$$\sigma_{ij}=\begin{cases}\sigma_i\sigma_j&i<j\\+1&i\ge j,\end{cases}$$
with
$$\sigma_i=\begin{cases}+1&i\le n\\-1&i>n,\end{cases}$$
for some fixed $n<N$, we get some kind of graded commutativity. The commutativity relation \eqref{eq.sigcom} becomes
$$u_{ij}u_{kl}=\sigma_i\sigma_j\sigma_k\sigma_l\,u_{kl}u_{ij}.$$
\end{ex}

\subsection{Constructing a partition category isomorphism}

In certain cases, it may happen that given a~compact matrix quantum group~$G$ such that $O_N^+\supset G\supset H_N$ corresponding to some partition category~$\Kat$, the deformation also satisfies $O_N^+\supset G^\sigma\supset H_N$ and hence is again described by a~linear category of partitions~$\tilde\Kat$ using the standard functor $p\mapsto T_p$ rather than $p\mapsto T_p^\sigma$.

This happens in the case of the $(-1)$-deformations. Indeed, taking $\sigma$ as in Example~\ref{ex.qdef} and $O_N^+\supset G\supset H_N$, we have
$$O_N^+=O_N^{+\sigma}\supset G^\sigma\supset H_N^\sigma=H_N.$$

It is easy to check the following
\begin{equation}
\label{eq.eqmap}
T_{\crosspart}^\sigma=-T_{\crosspart}+2T_{\Paaaa},\qquad T_{\Paaaa}^\sigma=T_{\Paaaa}.
\end{equation}
As a~consequence, we have that $O_N^{-1}$ is a~quantum group determined by the category of all pairings $\Pair_N=\langle\crosspart\rangle_N$ using the functor $p\mapsto T_p^\sigma$ or, equivalently, by the category $\langle\crosspart-2\Paaaa\rangle_N$ (which is isomorphic to $\Pair_N$ by Prop.~\ref{P.join}) using the standard functor $p\mapsto T_p$.

To put it in a different way: $O_N^{-1}$ is a twist of $O_N$. In order to describe its representation category using partitions, we have to either twist the functor $p\mapsto T_p$ or to twist the partition category itself.

\begin{rem}
In general, it is possible to show that there is an isomorphism of monoidal $*$\hbox{-}categories $\phi\colon\EvenPart_\delta\to\EvenPart_\delta$ mapping
$$\crosspart\mapsto -\crosspart+2\Paaaa,\qquad \Paaaa\mapsto\Paaaa.$$
Taking $\delta=N\in\N$, it holds that $T_p^\sigma=T_{\phi(p)}$ with~$\sigma$ as in Example~\ref{ex.qdef}.
\end{rem}

However, most easy categories are stable under this isomorphism. The interesting examples are the following ones.

\begin{prop}
\label{P.qcomex}
The following non-easy linear categories of partitions
$$\langle\crosspart-2\Paaaa\rangle_N,\qquad \langle\halflibpart-2\Pabaaba-2\Paabaab-2\Pabbabb+4\Paaaaaa\rangle_N$$
correspond to the quantum groups $O_N^{-1}=O_N^\sigma$ and $O_N^{*\,-1}=O_N^{*\,\sigma}$, respectively, where $\sigma$ comes from Example~\ref{ex.qdef}. That is, those are $(-1)$-deformations of the quantum groups $O_N$ and~$O_N^*$.
\end{prop}
\begin{proof}
It follows from the fact that the categories are images of $\Pair_N=\langle\crosspart\rangle_N$, resp.\ $\langle\halflibpart\rangle_N$ by the above defined isomorphism~$\phi$. See Section~\ref{secc.join}.
\end{proof}

\begin{rem}
One could obtain many examples of non-easy two-coloured categories by reformulating these results to the unitary case and applying them to the half-liberated two-coloured categories recently obtained in \cite{MW18,MW19}.
\end{rem}

Now, consider $\sigma$ as in Example~\ref{ex.grad}. Here, we can see that
$$O_N^+\supset O_N^\sigma\supset O_n\times O_{N-n}.$$
In particular, choosing $n=N-1$, we have
$$O_N^+\supset O_N^\sigma\supset O_n\times\hat\Z_2\simeq B_N'.$$

\begin{prop}
\label{P.qgradedex}
The non-easy category
$$\Kat=\langle\crosspart-{2\over N}(\Pabac+\Pabcb)+{4\over N}\Pabcd)\rangle_N$$
corresponds to the quantum group $G=U_{(N,\pm)}^*O_N^\sigma U_{(N,\pm)}$, where $\sigma$ is defined as in Example~\ref{ex.grad} for $n=N-1$ and $U_{(N,\pm)}$ is a~unitary matrix defined in \cite[Definition~4.5]{GW18}.
\end{prop}
\begin{proof}
It is straightforward to check that $T_p=U_{(N,\pm)}^*T_{\crosspart}^\sigma U_{(N,\pm)}$ for $p=\crosspart-{2\over N}(\Pabac+\Pabcb)+{4\over N}\Pabcd)$.
\end{proof}

Recall that we already showed in Proposition~\ref{P.disjoin} that this category is isomorphic to the category of pairings $\Pair_N=\langle\crosspart\rangle_N$.

\bibliographystyle{halpha}
\bibliography{mybase}

\end{document}